\documentclass[reqno]{amsart}
\usepackage[english]{babel}
\usepackage[utf8]{inputenc}
\usepackage{amsthm}
\usepackage{amssymb}
\usepackage{bbm}
\usepackage[usenames,dvipsnames,svgnames,table]{xcolor}
\usepackage{extarrows}
\usepackage{float}
\newtheoremstyle{Satzk}
{10pt}
{10pt}
{\itshape}
{}
{\bfseries}
{.}
{ }
{\thmname{#1} \thmnumber{#2}}
\newtheoremstyle{Satzl}
{10pt}
{10pt}
{\itshape}
{}
{\bfseries}
{.}
{ }
{\thmname{#1} \thmnumber{#2}: \thmnote{#3}}
\newtheoremstyle{Definitionk}
{10pt}
{10pt}
{}
{}
{\bfseries}
{.}
{ }
{\thmname{#1} \thmnumber{#2}}
\newtheoremstyle{Definitionl}
{10pt}
{10pt}
{}
{}
{\bfseries}
{.}
{ }
{\thmname{#1} \thmnumber{#2}: \thmnote{#3}}
\newtheoremstyle{Bemerkung}
{10pt}
{10pt}
{}
{}
{\bfseries}
{.}
{ }
{\thmname{#1}}
\theoremstyle{Definitionl}
\newtheorem{defi}{Definition}[section]

\theoremstyle{Definitionk}
\newtheorem{defk}[defi]{Definition}
\theoremstyle{definition}

\theoremstyle{Satzk}
\newtheorem{satzk}[defi]{Theorem}
\newtheorem{lemmak}[defi]{Lemma}

\theoremstyle{Satzl}
\newtheorem{satzl}[defi]{Theorem}

\newtheorem{lemmal}[defi]{Lemma}
\newtheorem{korl}[defi]{Corollary}
\theoremstyle{Bemerkung}
\newtheorem*{bem*}{Remark}
\newtheorem{bemk}[defi]{Remark}
\usepackage{xparse}
\makeatletter
\RenewDocumentCommand{\title}{om}{%
   \IfNoValueTF{#1}
     {\gdef\shorttitle{Local Well-Posedness for the DNLS in Besov Spaces}}
     {\gdef\shorttitle{#1}}%
   \gdef\@title{#2}%
}
\makeatother
\title{Local Well-Posedness for the Derivative Nonlinear Schrödinger Equation in Besov spaces}
\author{Cai Constantin Cloos}
\address{Universität Bielefeld, Fakultät für Mathematik, Postfach 100131, 33501 Bielefeld, Germany}
\email{ccloos@math.uni-bielefeld.de}
\begin{document}
\setcounter{page}{1}
\begin{abstract}
It is shown that the cubic derivative nonlinear Schrö\-din\-ger equation is locally well-posed in Besov spaces $B^{s}_{2,\infty}(\mathbb X)$, $s\ge\tfrac12$, where we treat the non-periodic setting $\mathbb X=\mathbb R$ and the periodic setting $\mathbb X=\mathbb T$ simultaneously. The proof is based on the strategy of Herr for initial data in $H^{s}(\mathbb T)$, $s\ge\tfrac12$. 
\end{abstract}
\maketitle
\section{Introduction and main result}
We study the Cauchy problem for the following derivative nonlinear Schrödinger equation:
\begin{align}
\begin{cases}
i\partial_tu+\partial_x^2u= i\partial_x(|u|^2u)+\lambda |u|^{2k}u\ &\,\text{in }\mathbb X \times(-T,T),\\
u(0)= u_0\ &\,\text{in }\mathbb X,
\end{cases}\label{G2}\end{align}
where $\lambda\in\mathbb R$, $k\in\mathbb N_0$, $T>0$, $\mathbb X=\mathbb R$ (\emph{non-periodic setting}) or $\mathbb X=\mathbb T:=\mathbb R/2\pi\mathbb Z$ (\emph{periodic setting}).
We look for solutions $u$ which satisfy the corresponding integral equation
\begin{align*}
u(t)= U_tu_0+\int_0^tU_{t-t'}\big[\partial_x(|u|^2u)(t')-i\lambda|u|^{2k}u(t')\big]\,\mathrm dt',\ t\in(-T,T),
\end{align*}
where $(U_tu_0)\widehat\ (\xi)= e^{-it\xi^2}\widehat{u_0}(\xi)
$ for $u_0\in\mathcal S(\mathbb X)$.\\[10pt]
In the non-periodic setting, Takaoka \cite{Ta2} showed local well-posedness for initial data $u_0\in H^s(\mathbb R)$ and $s\ge\tfrac12$ which improved the results of Hayashi and Ozawa \cite{HaOz2,Hay,HaOz} for initial data in $H^1(\mathbb R)$. The central tools have been Fourier restriction methods, local smoothing, Strichartz estimates and a gauge transformation which cancels out the unfavorable nonlinear term $2i|u|^2\partial_xu$.\\[10pt]
In the periodic setting, Herr \cite{Herr} showed local well-posedness in $H^s(\mathbb T)$, $s\ge\tfrac12$, by using an adapted gauge transformation and a suitable version of Bourgain's $L^4$-Strichartz estimate.\\[10pt]
For $s<\tfrac12$, Biagoni and Linares \cite{BiLi} showed that the flow map $u_0\mapsto u$ is no longer uniformly continuous. In this sense, $H^{1/2}$ is critical. However, with respect to scaling, $L^2$ is critical: If $u$ solves \eqref{G2} with initial datum $u_0$, then $u_\sigma(x,t):=\tfrac{1}{\sigma^{1/2}}u\big(\tfrac{x}{\sigma},\tfrac{t}{\sigma^2}\big)$, $\sigma>0$, is a solution for initial data $u_0\big(\tfrac{\cdot}{\sigma}\big)$ and we have $\|u_\sigma(t)\|_{L^2}=\big\|u\big(\tfrac{t}{\sigma^2}\big)\big\|_{L^2}$. In order to meet this gap between $L^2$ and $H^{1/2}$, Grünrock and Herr \cite{Grn2, GH} proved local well-posedness in spaces $\widehat{H^s_p}(\mathbb X)$, where 
\begin{align*}
\|f\|_{\widehat H^{s}_p}:=\big\|\widehat{J^{s}f}\big\|_{L^{p'}}
\end{align*}
for $s\ge\tfrac12$ and $1<p\le2$, $\tfrac1p+\tfrac1{p'}=1$. In the non-periodic setting, S. Guo, Ren and Wang \cite{SGuoDNLS} recently generalized this result to modulation spaces $M^s_{2,q}$ with
\begin{align*}
\|f\|_{M^s_{2,q}}:=\Big(\sum_{k\in\mathbb Z}\langle k\rangle^{sq}\big\|\mathbbm1_{[k-\frac12,k+\frac12]}\widehat f\big\|_{L^2}^q\Big)^{1/q}
\end{align*}
for $q\in[2,\infty)$ and $s\ge\tfrac12$. In the scaling sense, $M^{1/2}_{2,q}$ is subcritical for $2\le q<\infty$ and critical for $q=\infty$.\\[10pt]
We show local well-posedness for initial data in $B^s_{2,\infty}(\mathbb X)$, $s\ge\tfrac12$:
\begin{satzk}
\label{W}
Let $s\ge\tfrac12$ and $k\in\mathbb N_0$. For any $r>0$, there exists $T=T(r)>0$ and a metric space $M_{s,T}$, such that for all $u_0\in B_r:=\{f\in B_{2,\infty}^s:\ \|f\|_{B^s_{2,\infty}}<r\}$, the equation \eqref{G2} has a unique solution $u\in M_{s,T}\hookrightarrow \mathcal C([-T,T],B^s_{2,\infty})$. The flow map
\begin{align*}
\tilde F\colon B_r\to \mathcal C([-T,T],B^s_{2,\infty}),\ u_0\mapsto u
\end{align*}
is continuous.
\end{satzk}
Therefore, we point out that the method of \cite{Herr} is also applicable to the non-periodic setting with some slight modifications. We  extend this method to the Besov space setting by using several frequency-localization arguments. Noticing that $H^{1/2}(\mathbb X)\hookrightarrow B^{1/2}_{2,\infty}(\mathbb X)$, we improve the results of Takaoka \cite{Ta2} and Herr \cite{Herr}.\\[10pt]
Global well-posedness was shown by Hayashi and Ozawa \cite{HaOz2,Hay,HaOz} in the non-periodic setting for $u_0\in H^1$ with mass $\|u_0\|_{L^2}^2<2\pi$. For $\lambda=0$, Z. Guo and Wu \cite{Wu,GuWu} generalized this result to $u_0\in H^{1/2}$ with mass smaller than $4\pi$. Recently, Mosincat \cite{Mos} proved the same result in the periodic setting.\\[10pt]
There are also results for global weak solutions in Sobolev spaces corresponding to $H^1$ concerning Dirichlet and generalized periodic boundary conditions, compare for example \cite{Chen, Mesk}.\\[10pt]
The remainder of this paper is organized as follows: We complete this section with some general notation. In the second section, we briefly introduce the Gauge transformation and the Gauge equivalent Cauchy problem. In the third section, we establish the function spaces and basic estimates for the linear and the Duhamel term. In the fourth section, we prove the estimate for the trilinear derivative term $u^2\partial_x\overline u$. The fifth section treats the multilinear terms $|u|^{2k}u$ and $|u|^4u$. In the last section, we conclude local well-posedness for the Gauge equivalent problem which implies the statement of theorem \ref{W} by backward transformation.
\subsection*{Notation}
For $a,b\ge0$, we denote $a\lesssim b$ if $a\le cb\text{ for some $c>0$}$, $a\ll b$ if $Ca<b$ for a sufficiently large $C>1$ and $a\sim b$ if $C^{-1}a\le b\le Ca$ for a sufficiently large $C>1$. We write $\lesssim_\alpha$ if the implicit constants depends on a parameter $\alpha$.\\[10pt]
%
For measure spaces $\Omega_1,\Omega_2$ and product-measurable mappings $u\colon \Omega_1\times \Omega_2\to\mathbb C$, $(x,t)\mapsto u(x,t)$ such that $u(\,\cdot\,,t)\in X$ and $u(x,\,\cdot\,)\in Y$, we set
\begin{align*}
\|u\|_{Y_tX_x}:=\big\|t\mapsto\|x\mapsto u(x,t)\|_{X}\big\|_{Y}
\end{align*}
and shortly $\|u\|_{X_{t,x}}:=\|u\|_{X_tX_x}$ if $X=Y$.\\[10pt]
$\mathcal S(\mathbb R^n)$ denotes the space of all \emph{Schwartz functions} on $\mathbb R^n$, $\mathcal S(\mathbb T\times\mathbb R)$ the space of all functions $u\colon\mathbb R^2\to\mathbb C$ such that
\begin{align*}
u(x+2\pi,t)=u(x,t),\ u(\cdot,t)\in\mathcal C^\infty(\mathbb R),\ u(x,\cdot)\in\mathcal S(\mathbb R)
\end{align*}
and $\mathcal S(\mathbb T)$ stands for the space of $2\pi$-periodic $\mathcal C^\infty$-functions on $\mathbb R$.\\[10pt]
For $f\in\mathcal S(\mathbb R)$, we define the \emph{Fourier transform} $\widehat f$ via
\begin{align*}
\widehat{f}(\xi):=&\, \frac{1}{\sqrt{2\pi}}\int_{\mathbb R}e^{-ix\cdot\xi}f(x)\,\mathrm dx,\ \xi\in\mathbb R.
\end{align*}
For $f\in L^1(\mathbb T)$ and $g\in L^1(\mathbb Z)$, we denote 
\begin{align*}
\widehat{f}(\xi):=&\, \frac{1}{\sqrt{2\pi}}\int_{\mathbb T}e^{-ix\cdot \xi}f(x)\,\mathrm dx,\ \xi\in\mathbb Z.
\end{align*}
For $u\in\mathcal (\mathbb X\times\mathbb R)$, we set
\begin{align*}
\widehat u(\xi,\tau):=\frac{1}{2\pi}\int_{\mathbb R}\int_{\mathbb X}e^{i(x,t)\cdot(\xi,\tau)}u(x,t)\,\mathrm dx\,\mathrm dt.
\end{align*}
$J^{s}$ denotes the Bessel potential of order $-s$. This means
\begin{align*}
\widehat{J^sf}= \langle \cdot\rangle^s\widehat f,\ f\in\mathcal S(\mathbb X),
\end{align*}
where $\langle a\rangle:= (1+a^2)^{1/2},\ a\in\mathbb R.$\\[10pt]
$H^s(\mathbb X)$ is the \emph{Sobolev space} of order $s$ on $\mathbb X$, this means the completion of $\mathcal S(\mathbb X)$ with respect to the norm
\begin{align*}
\|f\|_{H^s}:=\|J^sf\|_{L^2}.
\end{align*}
For $u\in\mathcal S(\mathbb X\times\mathbb R)$ and $s\in\mathbb R$, we write
\begin{align*}
J^s_xu(x,t):=&\, J^s(u(t))(x),\\
\widehat{\Gamma_{\pm}^s}u(\xi,\tau):=&\, \langle\tau\pm\xi^2\rangle^{s}\widehat{u}(\xi,\tau)
\end{align*}
and shortly $\Gamma:=\Gamma_+$.\\[10pt]
We consider $\vec\xi= (\xi_1,...,\xi_n)$ and $\mathbb Y^n_\xi:= \{\vec\xi\in\mathbb Y^n:\ \sum_{j=1}^n\xi_j=\xi\}$, $\mathbb Y\in\{\mathbb R,\mathbb Z\}$. The convolution of functions $f_1,...,f_n$ on $\mathbb Y$ is written as
\begin{align*}
f_1\ast...\ast f_n(\xi)&= \int_{\mathbb Y^{n-1}}\prod_{j=1}^{n-1}f_j(\xi_j)\cdot f_n(\xi-\xi_1-...-\xi_{n-1})\,\mathrm d(\xi_1,...,\xi_{n-1})\\
&=:\int_{\mathbb Y^n_\xi}\prod_{j=1}^nf_j(\xi_j)\,\mathrm d\vec\xi,
\end{align*}
with integration with respect to the counting measure if $\mathbb Y=\mathbb Z$.\\[10pt]
Let $\chi\in\mathcal C^\infty(\mathbb R,[0,1])$ be radially decreasing such that $\chi=1$ on $[-1,1]$ and $\chi=0$ on $(-2,2)^\mathrm{c}$. For $T>0$, we introduce
\begin{align*}
\chi_T(\xi):=&\,\chi\left(\frac\xi T\right)-\chi\left(\frac{2\xi}{T}\right),\ 
\chi_{\le T}(\xi):=\chi\left(\frac\xi{T}\right).
\end{align*}
Hence $\mathrm{supp}\, \chi_T\subseteq\{\xi\in\mathbb R:\ \frac T2<|\xi|<2T\}$. 
In addition, let
\begin{align*}
\mathcal D:=&\,\{2^n:\ n\in\mathbb Z\}=\{N: N\text{ dyadic}\},\ 
\mathcal D_{1}:=\{N\ge 1:\ N\in\mathcal D\}.
\end{align*}
For $\xi\ne0$, there are not more than two $N\in\mathcal D$ such that $\chi_N(\xi)\ne0$. We have $\sum_{N\in\mathcal D}\chi_N(\xi)=1$ for all $\xi\ne0$ and $\chi_{\le1}(\xi)+\sum_{N\in\mathcal D}\chi_N(\xi)=1$ for all $\xi\in\mathbb R$. For $N\in\mathcal D_1$, $f\in\mathcal S(\mathbb X)$ and $u\in\mathcal S(\mathbb R\times\mathbb X)$, we denote
\begin{align*}
\widehat{P_Nf}(\xi):=&\, \begin{cases}\chi_N(\xi)\widehat f(\xi)\text{ for $N>1$},\\
\chi_{\le1}(\xi)\widehat f(\xi)(\xi)\text{ for $N=1$},
\end{cases}\\
P_Nu(x,t):=&\,P_N(u(t))(x).
\end{align*}
This means $\sum_{N\in\mathcal D_1}P_Nf=f$.\\[10pt]
Finally, let $\chi_{[0,1]}\in\mathcal C^{\infty}(\mathbb R)$ be a radially decreasing function satisfying $\chi_{[0,1]}=1$ on $[0,1]$ and $\chi_{[0,1]}=0$ on $(-1,2)^c$. For intervals $[a,b]$, we denote $\chi_{[a,b]}(\xi)=\chi_{[0,1]}(\tfrac{\xi-a}{b-a})$ and
\begin{align*}
P_{[a,b]}f(\xi):=&\, \chi_{[a,b]}(\xi)\widehat f(\xi),\ 
P_{[a,b]}u(x,t):= P_{[a,b]}(u(t))(x).
\end{align*}
\section{Gauge transformation}\label{Gauge}
We work with the gauge transformation as introduced by Hayashi and Ozawa \cite{HaOz2} for the non-periodic setting and adapted by Herr \cite{Herr} for the periodic setting.
\begin{defi}[Gauge transformation]
For 
\begin{align*}
J(f)(x):=&\,\int_{-\infty}^x|f(y)|^2\,\mathrm dy,\\
\mathcal J(f)(x):=&\,\frac{1}{2\pi}\int_0^{2\pi}\int_\vartheta^{x}\big(|f(y)|^2-\mu(f)\big)\,\mathrm dy\,\mathrm d\vartheta,\ 
\mu(f):=\tfrac1{2\pi}\|f\|_{L^2(\mathbb T)}^2,
\end{align*}
we define
\begin{align*}
G(f)(x):=&\, e^{-i J(f)(x)}f(x),\ f\in L^2(\mathbb R),\\
\mathcal G(f)(x):=&\, e^{-i\mathcal J(f)(x)}f(x),\ f\in L^2(\mathbb T),\\
G(u)(x,t):=&\, G(u(t))(x),\ u\in\mathcal C([-T,T],L^2(\mathbb R)),\\
\mathcal G(u)(x,t):=&\, \mathcal G(u(t))(x-2\mu(u)t),\ u\in\mathcal C(\mathbb [-T,T],L^2(\mathbb T)).
\end{align*}
\end{defi}
As shown in \cite{HaOz} and \cite{Herr}, we can consider the gauge equivalent problems
\begin{align*}
\begin{cases}
i\partial_tv+\partial_x^2v= -iv^2\partial_x\overline v-\frac12|v|^4v+\lambda |v|^{2k}v &\text{in }\mathbb R \times(-T,T),\\
v(0)= v_0 &\text{in }\mathbb R,
\end{cases}
\end{align*}
and
\begin{align*}
\begin{cases}
i\partial_tv+\partial_x^2v= -iv^2\partial_x\overline v-\frac12|v|^4v+\lambda|v|^{2k}v+\mu(v)|v|^2v-\psi(v)v &\text{in }\mathbb T \times(-T,T),\\
v(0)= v_0 &\text{in }\mathbb T,
\end{cases}
\end{align*}
where $\mu$ is defined as above and $\psi(v)(t):= \tfrac{1}{2\pi}\int_0^{2\pi}\big(2\mathrm{Im}(v\partial_x\overline v)(y,t)-\tfrac12|v|^4(y,t)\big)\,\mathrm dy+\mu(v)^2$. Denoting
\begin{align*}
\mathcal T(v)(t):= v(t)^2\partial_x\overline v(t),\ 
\mathcal Q(v)(t):= |v(t)|^4v(t)
\end{align*}
in the non-periodic setting and
\begin{align*}
\mathcal T(v)(t):=&\, v(t)^2\partial_x\overline v(t)-\frac{i}{2\pi}v(t)\int_0^{2\pi}2\mathrm{Im}(v\partial_x\overline v)(y,t)\,\mathrm dy,\\
\mathcal Q(v)(t):=&\, \Big(|v(t)|^4-\frac{1}{2\pi}\int_0^{2\pi}|v(t)|^4\,\mathrm dx\Big)v(t)\\
&\,\ -\frac1\pi\int_0^{2\pi}|v(t)|^2\,\mathrm dx\Big(|v(t)|^2-\frac{1}{2\pi}\int_0^{2\pi}|v(t)|^2\,\mathrm dx\Big)v(t)
\end{align*}
in the periodic setting leads to the Cauchy problem
\begin{align}
\begin{cases}
i\partial_tv+\partial_x^2v= -i\mathcal T(v)-\frac12\mathcal Q(v)+\lambda|v|^{2k}v\ &\,\text{in }\mathbb X \times(-T,T),\\
v(0)= v_0\ &\,\text{in }\mathbb X.
\end{cases}\label{TP2}\end{align}
In order to treat both settings simultaneously, the following characterizations of $\mathcal T$ and $\mathcal Q$ are helpful:
\begin{lemmak}
We have $\mathcal T(v)=\mathcal T(v,v,\overline v)$ and $\mathcal Q(v)=\mathcal Q(v,\overline v,v,\overline v,v)$, where
\begin{align*}
\mathcal T(v_1,v_2,v_3)\widehat{\textcolor[rgb]{1,1,1}{f}}(\xi,\tau)=&\, \frac{1}{2\pi}\int_{\mathbb R^3_\tau}\int_{\mathbb R^3_\xi}\widehat{v_1}(\xi_1,\tau_1)\widehat{v_2}(\xi_2,\tau_2)i\xi_3\widehat{{v_3}}(\xi_3,\tau_3)\,\mathrm d\vec\xi\,\mathrm d\vec\tau,\\
\mathcal Q(v_1,v_2,v_3,v_4,v_5)\widehat{\textcolor[rgb]{1,1,1}{f}}(\xi,\tau)=&\, \frac{1}{2\pi}\int_{\mathbb R^3_\tau}\int_{\mathbb R^3_\xi}\prod_{j=1}^5\widehat{v_j}(\xi_j,\tau_j)\,\mathrm d\vec\xi\,\mathrm d\vec\tau
\end{align*}
in the non-periodic setting and
\begin{align*}
\mathcal T(v_1,v_2,v_3)\widehat{\textcolor[rgb]{1,1,1}{f}}(\xi,\tau)
=&\, \frac{1}{(2\pi)^{3/2}}\int_{\mathbb R^3_\tau}\underset{\xi_1,\xi_2\ne\xi}{\sum_{\xi_1+\xi_2+\xi_3=\xi}}\widehat{v_1}(\xi_1,\tau_1)\widehat{v_2}(\xi_2,\tau_2)i\xi_3\widehat{{v_3}}(\xi_3,\tau_3)\,\mathrm d\vec\tau\nonumber\\
&\,\ +\frac{1}{(2\pi)^{3/2}}\int_{\mathbb R^3_\tau}\widehat{v_1}(\xi,\tau_1)\widehat{v_2}(\xi,\tau_2)i\xi\widehat{{v_3}}(-\xi,\tau_3)\,\mathrm d\vec\tau,\\
\mathcal Q(v_1,v_2,v_3,v_4,v_5)\widehat{\textcolor[rgb]{1,1,1}{f}}(\xi,\tau)
&\,= \frac{1}{(2\pi)^{5/2}}\int_{\mathbb R^3_\tau}\underset{\xi_1+\xi_2+\xi_3+\xi_4,\,\xi_1+\xi_2,\,\xi_3+\xi_4\ne0}{\sum_{\xi_1+\xi_2+\xi_3+\xi_4+\xi_5=\xi}}\prod_{j=1}^5\widehat{v_j}(\xi_j,\tau_j)\,\mathrm d\vec\tau
\end{align*}
in the periodic setting.
\end{lemmak}
\begin{proof}
In the non-periodic case, this is a direct consequence of the elementary properties of the Fourier transformation. For the periodic setting, compare \cite[Lemma 6.4]{GH}.
\end{proof}
\section{Spaces: Definition and basic properties}
\begin{defk}
For $s\in\mathbb R$, $p\in(0,\infty]$, $q\in(0,\infty]$, we define the \emph{Besov space} $B_{p,q}^s(\mathbb X)$ as the completion of $\mathcal S(\mathbb X)$ with respect to the norm
\begin{align*}
{\|f\|}_{B^s_{p,q}(\mathbb X)}:=\begin{cases}\|P_{1}f\|_{L^p(\mathbb X)}+\big(\sum_{N>1}N^{sq}\|P_Nf\|_{L^p(\mathbb X)}^q\big)^{1/q}&\,\text{ if }q<\infty,\\
\|P_{1}f\|_{L^p(\mathbb X)}+\sup_{N>1}N^s\|P_Nf\|_{L^p(\mathbb X)}&\,\text{ if }q=\infty,
\end{cases}
\end{align*}
where we take the supremum and the sum over dyadic numbers $N$. 
\end{defk}
\begin{defk}\label{D1}
Let $s,b\in\mathbb R$. We define $X^{s,b,\pm}(\mathbb X)$, $Y^{s,b}(\mathbb X)$, $Z^s(\mathbb X)$, $\mathfrak X^{s,b,\pm}(\mathbb X)$, $\mathcal Y^{s,b}(\mathbb X)$, $\mathcal Z^s(\mathbb X)$ as the completions of $\mathcal S(\mathbb X\times\mathbb R)$ with respect to the norms
\begin{align*}
{\|u\|}_{X^{s,b,\pm}}&\,:= {\|\left\langle\xi\right\rangle^s\left\langle\tau\pm\xi^2\right\rangle^b\widehat{u}(\xi,\tau)\|}_{L^2_\xi(\mathbb Y) L^2_\tau(\mathbb R)},\\
{\|u\|}_{Y^{s,b}}&\,:= {\|\left\langle\xi\right\rangle^s\left\langle\tau+\xi^2\right\rangle^b\widehat{u}(\xi,\tau)\|}_{L^2_\xi(\mathbb Y) L^1_\tau(\mathbb R)},\\
{\|u\|}_{Z^s}&\,:= {\|u\|}_{X^{s,\frac12}}+ {\|u\|}_{Y^{s,0}},\\
{\|u\|}_{\mathfrak X^{s,b,\pm}}&\,:= \|P_{1}u\|_{X^{s,b,\pm}}+\sup_{N>1}\|P_Nu\|_{X^{s,b,\pm}},\\
{\|u\|}_{\mathcal Y^{s,b}}&\,:= \|P_{1}u\|_{Y^{s,b}}+\sup_{N>1}\|P_Nu\|_{Y^{s,b}},\\
{\|u\|}_{\mathcal Z^{s}}&\,:= \|P_{1}u\|_{Z^{s}}+\sup_{N>1}\|P_Nu\|_{Z^{s}}.
\end{align*}
For $T>0$, we consider the space $\mathcal Z^s_T(\mathbb X):=\left\{u\big|_{[-T,T]}:\ u\in\mathcal Z^s(\mathbb X)\right\}$ with norm
\begin{align*}
{\|u\|}_{\mathcal Z^s_T}&\,:= \inf\left\{{\|v\|}_{Z^{s}}:\ u=v\big|_{[-T,T]},\, v\in\mathcal Z^s(\mathbb X)\right\}.
\end{align*}
\end{defk}
The following estimates for the linear term, the Duhamel term and for the behavior under multiplication with smooth cutoffs can be found e.g. in \cite{GrnDiss, Herr} for the case without frequency-localization. With trivial modifications, they remain true in our setting:
\begin{lemmak}
Let $s\ge 0$, $u_0\in B^s_{2,\infty}(\mathbb X)$ and $w\in\mathcal S(\mathbb X\times\mathbb R)$ such that $\mathrm{supp}\,w\subseteq\mathbb X\times[-2,2]$. Then
\begin{align*}
\|\chi(t)U_tu_0\|_{\mathcal Z^s}\lesssim&\,\|u_0\|_{B^s_{2,\infty}},\\
\Big\|\chi(t)\int_0^tU_{t-t'}w(t')\,\mathrm dt'\Big\|_{\mathcal Z^s}\lesssim&\,\|w\|_{\mathfrak X^{s,-\frac12}\cap\mathcal Y^{s,-1}}.
\end{align*}
For $u\in\mathcal S(\mathbb X\times\mathbb R)$, $s\in\mathbb R$, $0\le b_1<b_2<\tfrac12$, ${N}\in\mathcal D_1$, $\delta>0$ and $T\in(0,1]$, we have
\begin{align*}
\|P_N(\chi_T(t)u)\|_{Y^{s,0}}\lesssim&\, \|P_Nu\|_{Y^{s,0}},\\
\|P_N(\chi_T(t)u)\|_{X^{s,\frac12}}\lesssim&\, T^{-\delta}\|P_Nu\|_{X^{s,\frac12}},\\
\|P_N(\chi_T(t)u)\|_{X^{s,b_1,\pm}}\lesssim&\, T^{b_2-b_1}\|P_Nu\|_{X^{s,b_2,\pm}}.
\end{align*}
For $T>0$, the embedding $\mathcal Z^s_T\hookrightarrow\, \mathcal C([-T,T],B^s_{2,\infty})$ holds true.
\end{lemmak}
The following statements can be found again in \cite{GrnDiss,Herr}:
\begin{lemmak}
Let $u\in\mathcal S(\mathbb X\times \mathbb R)$. Then
\begin{align}
\|u\|_{Y^{s,b_1}}\lesssim&\, \|u\|_{X^{s,b_2}}\ \ \forall\,b_2>b_1+\frac12,\label{XY}\\
\|u\|_{L^p_tL^q_x}\lesssim&\, \|u\|_{X^{s,b,\pm}}\ \ \forall\,p,q\in[2,\infty),\ b\ge\frac12-\frac1p,\ s\ge\frac12-\frac1q,\label{11}\\
\|u\|_{X^{s,b,\pm}}\lesssim&\, \|u\|_{L^p_tH^s_x}\ \ \forall\,p\in(1,2],\ b\le \frac12-\frac1p.\label{12}
\end{align}
\end{lemmak}
\begin{lemmal}[Strichartz Estimates]
Let $u\in\mathcal S(\mathbb X\times\mathbb R)$. We have
\begin{align}
\|u\|_{L^4_{t,x}}\lesssim&\, \|u\|_{X^{0,b,\pm}}\ \ \forall\,b>\frac38,\label{L41}\\
\|u\|_{X^{0,b,\pm}}\lesssim&\, \|u\|_{L^{4/3}_{t,x}}\ \ \forall\,b<-\frac38\label{L43}.
\end{align}
For $\mathbb X=\mathbb R$ and $b>\tfrac12$, $p\in(2,\infty]$, $q\in[2,\infty]$ satisfying $\tfrac2p+\tfrac1q=\tfrac12$, it also holds that
\begin{align}
\|u\|_{L^p_tL^q_x}\lesssim&\, \|u\|_{X^{0,b,\pm}}\label{LpLq}.
\end{align}
Finally, for $\tilde p$ with $\tfrac1{\tilde p}= \vartheta\tfrac16+(1-\vartheta)\tfrac12$, $\vartheta\in(0,1)$, $b>\tfrac12$, we have
\begin{align}
\|u\|_{L^{\tilde p}_{t,x}}\lesssim \|u\|_{X^{0,\vartheta b,\pm}}.\label{lplq}
\end{align}
\end{lemmal}
\begin{proof}
Since $\|\cdot\|_{L^p_tL^q_x}$ is invariant under complex conjugation, it suffices to consider $\|\cdot\|_{X^{s,b,+}}$. In the periodic setting, \eqref{L41} and \eqref{L43} have been shown in \cite[Lemma 2.1]{Grn}. Since $(p,q)$ is a Strichartz pair, we obtain \eqref{LpLq}, compare for example \cite[Lemma 2.3]{GTV}. Estimate \eqref{lplq} can be concluded by interpolation (compare for example \cite[Lemma 1.4]{GrnDiss}) between \eqref{LpLq} and the trivial statement
\begin{align*}
\|u\|_{L^2_{t,x}}=\|u\|_{X^{0,0}}.
\end{align*}
In the non-periodic setting, \eqref{L41} is a direct consequence of \eqref{lplq} by plugging in $\vartheta=\tfrac34$. Finally, estimate \eqref{L43} follows from \eqref{L41} and duality.
\end{proof}
\section{Trilinear estimate}
In this section, we handle the trilinear term $\mathcal T(u)$ which is essentially $u^2\partial_x\overline u$. The partial derivative on the third factor causes a factor $\xi_3$ on the Fourier side. If $|\xi_3|$ is significantly higher than the first two frequencies, we can use the following \emph{resonance relation} to control the derivative term: For $(\xi,\tau)\in\mathbb R^2$, $(\xi_1,\xi_2,\xi_3)$, $(\tau_1,\tau_2,\tau_3)\in\mathbb R^3$ such that $\xi_1+\xi_2+\xi_3=\xi$ and $\tau_1+\tau_2+\tau_3=\tau$, we have
\begin{align}
4\max\{|\tau+\xi^2|,&\,|\tau_1+\xi_1^2|,|\tau_2+\xi_2^2|,|\tau_3-\xi_3^2|\}\nonumber\\
\ge&\,|\tau+\xi^2-(\tau_1+\xi_1^2+\tau_2+\xi_2^2+\tau_3-\xi_3^2)|\nonumber\\
=&\,2|\xi_1+\xi_3\|\xi_2+\xi_3|.\label{Resona}
\end{align}
For $|\xi_3|\gg|\xi_1|,|\xi_2|$, we can conclude $\max\{|\tau+\xi^2|,|\tau_1+\xi_1^2|,|\tau_2+\xi_2^2|,|\tau_3-\xi_3^2|\}\gtrsim \xi_3^2$.\\[10pt]
In the sequel, we consider the multipliers of \cite{Herr} with some slight modifications:
\begin{defk}\label{M}
Let $j\in\{1,2,3\},\ \xi\in\mathbb Y,\ \tau\in\mathbb R,\ \xi_j\in\mathbb Y,\ \tau_j\in\mathbb R,\ \vec\xi=(\xi_1,\xi_2,\xi_3)$, $\vec\tau=(\tau_1,\tau_2,\tau_3)$. We set $A(\xi,\tau,\vec\xi,\vec\tau):=\{|\tau+\xi^2|,|\tau_1+\xi_1^2|,|\tau_2+\xi_2^2|,|\tau_3-\xi_3^2|\}$,
\begin{align*}
A_0(\xi,\tau):=&\,\big\{(\vec\xi,\vec\tau)\in\mathbb Y^3_\xi\times\mathbb R^3_\tau:\,\max A(\xi,\tau,\vec\xi,\vec\tau)=|\tau+\xi^2|\big\},\\
A_1(\xi,\tau):=&\,\big\{(\vec\xi,\vec\tau)\in\mathbb Y^3_\xi\times\mathbb R^3_\tau:\,\max A(\xi,\tau,\vec\xi,\vec\tau)=|\tau_1+\xi_1^2|\big\},\\
A_2(\xi,\tau):=&\,\big\{(\vec\xi,\vec\tau)\in\mathbb Y^3_\xi\times\mathbb R^3_\tau:\,\max A(\xi,\tau,\vec\xi,\vec\tau)=|\tau_2+\xi_2^2|\big\},\\
A_3(\xi,\tau):=&\,\big\{(\vec\xi,\vec\tau)\in\mathbb Y^3_\xi\times\mathbb R^3_\tau:\,\max A(\xi,\tau,\vec\xi,\vec\tau)=|\tau_3-\xi_3^2|\big\}
\end{align*}and
\begin{align*}
&M(\xi,\tau,\vec\xi,\vec\tau):= \frac{\langle\xi\rangle^{1/2}i\xi_3}{\langle\tau+\xi^2\rangle^{1/2}\langle\tau_1+\xi_1^2\rangle^{1/2}\langle\tau_2+\xi_2^2\rangle^{1/2}\langle\tau_3-\xi_3^2\rangle^{1/2}\langle\xi_1\rangle^{1/2}\langle\xi_2\rangle^{1/2}\langle\xi_3\rangle^{1/2}},\\
&M_0(\xi,\tau,\vec\xi,\vec\tau):= \frac{\mathbbm1_{A_0(\xi,\tau)}(\vec\xi,\vec\tau)}{\langle\tau_1+\xi_1^2\rangle^{1/2}\langle\tau_2+\xi_2^2\rangle^{1/2}\langle\tau_3-\xi_3^2\rangle^{1/2}\langle\xi_1\rangle^{1/2}\langle\xi_2\rangle^{1/2}},\\
&M_1(\xi,\tau,\vec\xi,\vec\tau):= \frac{\mathbbm1_{A_1(\xi,\tau)}(\vec\xi,\vec\tau)}{\langle\tau+\xi^2\rangle^{1/2}\langle\tau_2+\xi_2^2\rangle^{1/2}\langle\tau_3-\xi_3^2\rangle^{1/2}\langle\xi_1\rangle^{1/2}\langle\xi_2\rangle^{1/2}},\\
&M_2(\xi,\tau,\vec\xi,\vec\tau):= \frac{\mathbbm1_{A_2(\xi,\tau)}(\vec\xi,\vec\tau)}{\langle\tau+\xi^2\rangle^{1/2}\langle\tau_1+\xi_1^2\rangle^{1/2}\langle\tau_3-\xi_3^2\rangle^{1/2}\langle\xi_1\rangle^{1/2}\langle\xi_2\rangle^{1/2}},\\
&M_3(\xi,\tau,\vec\xi,\vec\tau):= \frac{\mathbbm1_{A_3(\xi,\tau)}(\vec\xi,\vec\tau)}{\langle\tau+\xi^2\rangle^{1/2}\langle\tau_1+\xi_1^2\rangle^{1/2}\langle\tau_2+\xi_2^2\rangle^{1/2}\langle\xi_1\rangle^{1/2}\langle\xi_2\rangle^{1/2}},\\
&M_4(\xi,\tau,\vec\xi,\vec\tau):= \frac{1}{\langle\tau+\xi^2\rangle^{7/16}\langle\tau_1+\xi_1^2\rangle^{7/16}\langle\tau_2+\xi_2^2\rangle^{7/16}\langle\tau_3-\xi_3^2\rangle^{7/16}},\\
&\tilde M(\xi,\tau,\vec\xi,\vec\tau):= \frac{M(\xi,\tau,\vec\xi,\vec\tau)}{\langle\tau+\xi^2\rangle^{1/2}},\\
&\tilde M_0(\xi,\tau,\vec\xi,\vec\tau):= \frac{\mathbbm1_{A_0(\xi,\tau)}(\vec\xi,\vec\tau)}{\langle\tau_1+\xi_1^2\rangle^{\frac12+\delta}\langle\tau_2+\xi_2^2\rangle^{\frac12+\delta}\langle\tau_3-\xi_3^2\rangle^{\frac12+\delta}\langle\xi\rangle^{\frac12-3\delta}\langle\xi_1\rangle^{\frac12}\langle\xi_2\rangle^{\frac12}\langle\xi_3\rangle^{\frac12-3\delta}},\\
&\tilde M_j(\xi,\tau,\vec\xi,\vec\tau):= \frac{M_j(\xi,\tau,\vec\xi,\vec\tau)}{\langle\tau+\xi^2\rangle^{1/2}},\ j\in\{1,2,3,4\},
\end{align*}
where we will choose a $\delta\in(0,\tfrac16)$.
\end{defk}
\begin{lemmak}
Let $u_j\in\mathcal S(\mathbb X\times\mathbb R)$ such that $\mathrm{supp}\,u_j\subseteq\mathbb X\times[-T,T]$, $T\in(0,1]$, $j\in\{1,2,3\}$, and $f_j(\xi,\tau):=\langle\tau+\xi^2\rangle^{1/2}\langle\xi\rangle^{1/2}\widehat{u_j}(\xi,\tau)$ for $j\in\{1,2\}$ and $f_3(\xi,\tau):=\langle\tau-\xi^2\rangle^{1/2}\langle\xi\rangle^{1/2}\widehat{u_3}(\xi,\tau)$. We have
\begin{align}
|M|\lesssim \sum_{j=0}^4M_j,\ 
|\tilde M|\lesssim \sum_{j=0}^4\tilde{M}_j\label{tm}
\end{align}
and
\begin{align}
\Big\|\int_{\mathbb R^3_\tau}\int_{\mathbb Y^3_\xi}M_0(\xi,\tau,\vec\xi,\vec\tau)&\,f_{1}(\xi_1,\tau_1)f_{2}(\xi_2,\tau_2)f_{3}(\xi_3,\tau_3)\,\mathrm d\vec\xi\,\mathrm d\vec\tau\Big\|_{L^2_{\xi,\tau}}\nonumber\\
\lesssim&\, \|u_1\|_{X^{\frac38,\frac38}}\|u_2\|_{X^{\frac38,\frac38}}\|u_3\|_{X^{\frac12,\frac12,-}}\label{M0},\\
\Big\|\int_{\mathbb R^3_\tau}\int_{\mathbb Y^3_\xi}M_1(\xi,\tau,\vec\xi,\vec\tau)&\,f_{1}(\xi_1,\tau_1)f_{2}(\xi_2,\tau_2)f_{3}(\xi_3,\tau_3)\,\mathrm d\vec\xi\,\mathrm d\vec\tau\Big\|_{L^2_{\xi,\tau}}\nonumber\\
\lesssim&\, \|u_1\|_{X^{\frac38,\frac12}}\|u_2\|_{X^{\frac38,\frac38}}\|u_3\|_{X^{\frac12,\frac12,-}}
\label{M1},\\
\Big\|\int_{\mathbb R^3_\tau}\int_{\mathbb Y^3_\xi}M_2(\xi,\tau,\vec\xi,\vec\tau)&\,f_{1}(\xi_1,\tau_1)f_{2}(\xi_2,\tau_2)f_{3}(\xi_3,\tau_3)\,\mathrm d\vec\xi\,\mathrm d\vec\tau\Big\|_{L^2_{\xi,\tau}}\nonumber\\
\lesssim&\, \|u_1\|_{X^{\frac38,\frac38}}\|u_2\|_{X^{\frac38,\frac12}}\|u_3\|_{X^{\frac12,\frac12,-}}
\label{M2},\\
\Big\|\int_{\mathbb R^3_\tau}\int_{\mathbb Y^3_\xi}M_3(\xi,\tau,\vec\xi,\vec\tau)&\,f_{1}(\xi_1,\tau_1)f_{2}(\xi_2,\tau_2)f_{3}(\xi_3,\tau_3)\,\mathrm d\vec\xi\,\mathrm d\vec\tau\Big\|_{L^2_{\xi,\tau}}\nonumber\\
\lesssim&\, \|u_1\|_{X^{\frac38,\frac38}}\|u_2\|_{X^{\frac38,\frac38}}\|u_3\|_{X^{\frac12,\frac12,-}}
\label{M3},\allowdisplaybreaks[1]\\
\Big\|\int_{\mathbb R^3_\tau}\int_{\mathbb Y^3_\xi}M_4(\xi,\tau,\vec\xi,\vec\tau)&\,f_{1}(\xi_1,\tau_1)f_{2}(\xi_2,\tau_2)f_{3}(\xi_3,\tau_3)\,\mathrm d\vec\xi\,\mathrm d\vec\tau\Big\|_{L^2_{\xi,\tau}}\nonumber\\
\lesssim&\, \|u_1\|_{X^{\frac12,\frac{15}{32}}}\|u_2\|_{X^{\frac12,\frac{15}{32}}}\|u_3\|_{X^{\frac12,\frac{15}{32},-}}.\label{N}\allowdisplaybreaks[1]
\end{align}
\end{lemmak}
\begin{proof}
In the periodic setting, these statements have already been shown in \cite[Lemma 4.1, Lemma 4.2, Thm. 4.1]{Herr}. In that case, for \eqref{tm}, we only had to consider $\vec\xi\in\mathbb Z^3$. But in the non-periodic setting we need to consider $\vec\xi\in\mathbb R^3$. This means, we have to modify the proof of \cite{Herr} slightly. For the sake of completeness, we will show all these statements simultaneously for the periodic and non-periodic setting.\\[10pt]
Let $(\xi,\tau)\in\mathbb Y\times\mathbb R$, $\vec\xi\in\mathbb Y^3_\xi$ and $\vec\tau\in\mathbb R^3_\tau$. Since $\tau+\xi^2-(\tau_1+\xi_1^2+\tau_2+\xi_2^2+\tau_3-\xi_3^2)=2(\xi-\xi_1)(\xi-\xi_2)$, an application of the triangle inequality shows that
\begin{align}
\langle(\xi-\xi_1)(\xi-\xi_2)\rangle^{1/2}
\le 4\Big(\mathbbm1_{A_0(\xi,\tau)}(\vec\xi,\vec\tau)&\,\left\langle\tau+\xi^2\right\rangle^{1/2}+\mathbbm1_{A_1(\xi,\tau)}(\vec\xi,\vec\tau)\left\langle\tau_1+\xi_1^2\right\rangle^{1/2}\nonumber\\
&\,+\mathbbm1_{A_2(\xi,\tau)}(\vec\xi,\vec\tau)\left\langle\tau_2+\xi_2^2\right\rangle^{1/2}\nonumber\\
&\,+\mathbbm1_{A_3(\xi,\tau)}(\vec\xi,\vec\tau)\left\langle\tau_3-\xi_3^2\right\rangle^{1/2}\Big)\label{A24}.
\end{align}
We consider the following four cases:\\[10pt]
\emph{(i) $|\xi|>2|\xi_1|$ and $|\xi|>2|\xi_2|$}: Here, $|\xi_3|\lesssim |\xi|$ and $\langle(\xi-\xi_1)(\xi-\xi_2)\rangle\gtrsim \left\langle\xi\right\rangle^2$.
Hence, by \eqref{A24},
\begin{align*}
|M(\xi,\tau,\vec\xi,\vec\tau)|\lesssim&\,\sum_{j=0}^3M_j(\xi,\tau,\vec\xi,\vec\tau).
\end{align*}
\emph{(ii) $|\xi|\le2|\xi_1|$ and $|\xi|\le2|\xi_2|$}: In this case, we have $
|\xi_3|\lesssim\max\{|\xi_1|,|\xi_2|\}$ and $|\xi|\le 2\min\{|\xi_1|,|\xi_2|\}$. This means
\begin{align*}
|M(\xi,\tau,\vec\xi,\vec\tau)|\lesssim&\,M_4(\xi,\tau,\vec\xi,\vec\tau).
\end{align*}
\emph{(iii) $|\xi|>2|\xi_1|$ and $\xi\le2|\xi_2|$}: Since $|\xi|\le |\xi-\xi_1|+\frac12|\xi|$, we have $|\xi|\le 2|\xi-\xi_1|$ and therefore $2\langle(\xi-\xi_1)(\xi-\xi_2)\rangle\ge |\xi|\cdot|\xi-\xi_2|$. Furthermore, $
\langle \xi\rangle^{1/2}\le 1+|\xi|^{1/2}$ and $|\xi_3|^{1/2}\le |\xi-\xi_2|^{1/2}+|\xi_1|^{1/2}$. Using $\langle\xi_1\rangle\ge1$ and \eqref{A24} provides
\begin{align*}
|M(\xi,\tau,\vec\xi,\vec\tau)|\lesssim&\, \sum_{j=0}^4M_j(\xi,\tau,\vec\xi,\vec\tau).
\end{align*}
\emph{(iv) $|\xi|\le2|\xi_1|$ and $|\xi|>2|\xi_2|$}: By symmetry, this is a direct consequence of $(iii)$.\\[10pt]
From the first estimate of \eqref{tm}, we obtain 
\begin{align*}
|\tilde M(\xi,\tau,\vec\xi,\vec\tau)|\lesssim&\, \langle\tau+\xi^2\rangle^{-1/2}M_0(\xi,\tau,\vec\xi,\vec\tau)+\sum_{j=1}^4\tilde M_j(\xi,\tau,\vec\xi,\vec\tau).
\end{align*}
This means, the second estimate of \eqref{tm} follows from $\langle\tau+\xi^2\rangle^{-1/2}M_0(\xi,\tau,\vec\xi,\vec\tau)\lesssim \tilde M_0(\xi,\tau,\vec\xi,\vec\tau)$. Therefore, we consider again the four cases from above:\\[10pt]
\emph{(i) $|\xi|>2\xi_1$ and $|\xi|>2|\xi_2|$}: Here, $|\xi_3|\lesssim|\xi|$. For $(\vec\xi,\vec\tau)\in A_0(\xi,\tau)$, resonance relation \eqref{Resona} implies $|\tau+\xi^2|\gtrsim |\xi|\cdot|\xi_3|$. Hence
\begin{align*}
\langle\tau+\xi^2\rangle^{1/2}\gtrsim \langle\tau_1+\xi_1^2\rangle^{\delta}\langle\tau_2+\xi_2^2\rangle^{\delta}\langle\tau_3-\xi_3^2\rangle^{\delta}\langle\xi\rangle^{\frac12-3\delta}\langle\xi_3\rangle^{\frac12-3\delta}
\end{align*}
and consequently $\langle\tau+\xi^2\rangle^{-1/2}M_0(\xi,\tau,\vec\xi,\vec\tau)\lesssim \tilde M_0(\xi,\tau,\vec\xi,\vec\tau)$.\\[10pt]
\emph{(ii) $|\xi|\le2|\xi_1|$ and $|\xi|\le2|\xi_2|$}: In this case, we have $|\xi_3|\lesssim\max\{|\xi_1|,|\xi_2|\}$ and $|\xi|\lesssim 2\min\{|\xi_1|,|\xi_2|\}$ which implies $|\tilde M|\lesssim \tilde M_4$.\\[10pt]
\emph{(iii) $|\xi|>2|\xi_1|$ and $|\xi|\le 2|\xi_2|$}: First, let $|\xi|<1$. Then $|\xi_1|<\tfrac12$ and $|\xi_2+\xi_3|<\tfrac32$. This means $|\xi_3|< \tfrac32+|\xi_2|$ and $\langle\xi_3\rangle\lesssim\langle\xi_2\rangle$. Since $|\xi|<1,|\xi_1|<\tfrac12$, we obtain
\begin{align*}
|\tilde M(\xi,\tau,\vec\xi,\vec\tau)|\lesssim \tilde M_4(\xi,\tau,\vec\xi,\vec\tau).
\end{align*}
Secondly, let $|\xi_3|<1$. Then $\langle\xi_3\rangle\sim 1$ and from $|\xi|\le2|\xi_2|$, $\langle\xi_1\rangle\ge1$, we obtain
\begin{align*}
|\tilde M(\xi,\tau,\vec\xi,\vec\tau)|\lesssim&\, \tilde M_4(\xi,\tau,\vec\xi,\vec\tau).
\end{align*}
Thirdly, assume that $|\xi_1|>|\xi-\xi_2|$. Then $|\xi_3|\le |\xi-\xi_2|+|\xi_1|<2|\xi_1|$ and $|\xi|\le 2|\xi_2|$ which implies
\begin{align*}
|\tilde M(\xi,\tau,\vec\xi,\vec\tau)|\lesssim \tilde M_4(\xi,\tau,\vec\xi,\vec\tau).
\end{align*}
Finally, we have to consider the case $|\xi|,|\xi_3|\ge1$ and $|\xi_1|\le|\xi-\xi_2|$. Here, $\langle\xi\rangle\sim|\xi|$,\,$\langle\xi_3\rangle\sim|\xi_3|$ and the triangle inequality provides
\begin{align*}
|\xi_3|\le |\xi-\xi_2|+|\xi_1|\le 2|\xi-\xi_2|.
\end{align*}
Furthermore, we have
\begin{align*}
|\xi|\le&\, 2|\xi-\xi_1|,\ 
|\xi|\cdot|\xi-\xi_2|\le 2\langle(\xi-\xi_1)(\xi-\xi_2)\rangle
\end{align*}
and as a consequence $
\langle\xi\rangle\langle\xi_3\rangle\lesssim\langle(\xi-\xi_1)(\xi-\xi_2)\rangle$. 
For $(\vec\xi,\vec\tau)\in A_0(\xi,\tau)$, resonance relation \eqref{Resona} implies
\begin{align*}
\langle\tau+\xi^2\rangle^{1/2}
\gtrsim&\,\langle(\xi-\xi_1)(\xi-\xi_2)\rangle^{\frac12-3\delta}\langle\tau_1+\xi_1^2\rangle^{\delta}\langle\tau_2+\xi_2^2\rangle^{\delta}\langle\tau_3-\xi_3^2\rangle^{\delta}\\
\gtrsim&\, \langle\xi\rangle^{\frac12-3\delta}\langle\xi_3\rangle^{\frac12-3\delta}\langle\tau_1+\xi_1^2\rangle^{\delta}\langle\tau_2+\xi_2^2\rangle^{\delta}\langle\tau_3-\xi_3^2\rangle^{\delta},
\end{align*}
so that $\langle\tau+\xi^2\rangle^{-1/2}M_0(\xi,\tau,\vec\xi,\vec\tau)\lesssim \tilde M_0(\xi,\tau,\vec\xi,\vec\tau)$.\\[10pt]
\emph{(iv) $|\xi|\le2|\xi_1|$ and $|\xi|>2|\xi_2|$}: This is again a consequence of case $(iii)$.\\[10pt]
Now, we prove \eqref{M0}-\eqref{N}: By definition of $M_0$, Hölder's inequality and estimates \eqref{11},\,\eqref{L41}, we obtain
\begin{align*}
\Big\|\int_{\mathbb R^3_\tau}\int_{\mathbb Y^3_\xi}M_0(\xi,\tau,\vec\xi,\vec\tau)\prod_{j=1}^3f_j(\xi_j,\tau_j)\,\mathrm d\vec\xi\,\mathrm d\vec\tau\Big\|_{L^2_{\xi,\tau}}\le&\, \|u_1\cdot u_2\cdot J_x^{1/2}u_3\|_{L^2_{t,x}}\\
\lesssim&\, \|u_1\|_{X^{\frac38,\frac38}}\|u_2\|_{X^{\frac38,\frac38}}\|u_3\|_{X^{\frac12,\frac12,-}}.
\end{align*}
For $M_1$, we denote
\begin{align*}
\Big\|\int_{\mathbb R^3_\tau}\int_{\mathbb Y^3_\xi}M_1(\xi,\tau,\vec\xi,\vec\tau)\prod_{j=1}^3f_j(\xi_j,\tau_j)\,\mathrm d\vec\xi\,\mathrm d\vec\tau\Big\|_{L^2_{\xi,\tau}}\le \|\Gamma^{1/2}u_1\cdot u_2\cdot J_x^{1/2}u_3\|_{X^{0,-\frac12}}.
\end{align*}
Applying \eqref{12}, Hölder's inequality, \eqref{11} and \eqref{L41} provides
\begin{align}
\|\Gamma^{1/2}u_1\cdot u_2\cdot J_x^{1/2}u_3\|_{X^{0,-\frac38}}\lesssim&\, \|\Gamma^{1/2}u_1\cdot u_2\cdot J_x^{1/2}u_3\|_{L^{8/7}_tL^2_x}\nonumber\\
\le&\, \|\Gamma^{1/2}u_1\|_{L^2_tL^8_x}\|u_2\|_{L^8_{t,x}}\|J_x^{1/2}u_3\|_{L^4_{t,x}}\nonumber\\
\lesssim&\, \|u_1\|_{X^{\frac38,\frac12}}\|u_2\|_{X^{\frac38,\frac38}}\|u_3\|_{X^{\frac12,\frac12,-}}.\label{m11}
\end{align}
Since $\|\cdot\|_{X^{0,-\frac12}}\le\|\cdot\|_{X^{0,-\frac38}}$, we obtain the estimate for $M_1$.\\[10pt]
Because of symmetry, the estimate for $M_2$ can be shown analogously.\\[10pt]
By definition of $M_3$,
\begin{align*}
\Big\|\int_{\mathbb R^3_\tau}\int_{\mathbb Y^3_\xi}M_3(\xi,\tau,\vec\xi,\vec\tau)\prod_{j=1}^3f_{j}(\xi_j,\tau_j)\,\mathrm d\vec\xi\,\mathrm d\vec\tau\Big\|_{L^2_{\xi,\tau}}\le \|u_1\cdot u_2\cdot J_x^{1/2}\Gamma_-^{1/2}u_3\|_{X^{0,-\frac12}}.
\end{align*}
Using \eqref{L43}, Hölder's inequality and \eqref{11} provides
\begin{align}
\|u_1\cdot u_2\cdot J_x^{1/2}\Gamma_-^{1/2}u_3\|_{X^{0,-\frac{7}{16}}}\lesssim&\, \|u_1\cdot u_2\cdot J_x^{1/2}\Gamma_-^{1/2}u_3\|_{L^{4/3}_{t,x}}\nonumber\\
\lesssim&\, \|u_1\|_{X^{\frac38,\frac38}}\|u_2\|_{X^{\frac38,\frac38}}\|u_3\|_{X^{\frac12,\frac12,-}}\label{m33}.
\end{align}
Since $\|\cdot\|_{X^{0,-\frac12}}\le\|\cdot\|_{X^{0,-\frac7{16}}}$, we obtain the conclusion for $M_3$.\\[10pt]
For $M_4$, we have
\begin{align*}
\Big\|\int_{\mathbb R^3_\tau}\int_{\mathbb Y^3_\xi}M_4(\xi,\tau,\vec\xi,\vec\tau)&\prod_{j=1}^3f_{j}(\xi_j,\tau_j)\,\mathrm d\vec\xi\,\mathrm d\vec\tau\Big\|_{L^2_{\xi,\tau}}\\
\le&\, \big\|J_x^{1/2}\Gamma^{1/16}u_1\cdot J_x^{1/2}\Gamma^{1/16}u_2\cdot J_x^{1/2}\Gamma_-^{1/16}u_3\big\|_{X^{0,-\frac{7}{16}}}.
\end{align*}
The dual Strichartz estimate \eqref{L43}, Hölder's inequality and Strichartz estimate \eqref{L41} imply
\begin{align}
\big\|J_x^{1/2}\Gamma^{1/16}u_1\cdot J_x^{1/2}\,&\Gamma^{1/16}u_2\cdot J_x^{1/2}\Gamma_-^{1/16}u_3\big\|_{X^{0,-\frac{7}{16}}}\nonumber\\
\lesssim&\, \big\|J_x^{1/2}\Gamma^{1/16}u_1\cdot J_x^{1/2}\Gamma^{1/16}u_2\cdot J_x^{1/2}\Gamma_-^{1/16}u_3\big\|_{L^{4/3}_{t,x}}\nonumber\\
\lesssim&\,\|u_1\|_{X^{\frac12,\frac{15}{32}}}\|u_2\|_{X^{\frac12,\frac{15}{32}}}\|u_3\|_{X^{\frac12,\frac{15}{32},-}}\label{m44}
\end{align}
which proves the estimate for $M_4$.
\end{proof}
\begin{bemk}Similar estimates for $\tilde M_j$ will be shown in the proof of theorem \ref{Tri}.\end{bemk}
To prove the trilinear estimate, we follow the ideas of \cite[Thm. 4.1 and Thm. 4.2]{Herr} for the $H^s$-case. In order to localize frequencies, we write $u_j=\sum_{N_j\in\mathcal D_1}P_{N_j}u_j$. We need to eliminate the sums over $N_j\in\mathcal D_1$. Therefore, the following elementary estimates are useful:
\begin{lemmak}
Let $u\in\mathcal S(\mathbb X\times\mathbb R)$, $\delta>0$ and $s,b\in\mathbb R$. Then
\begin{align}
\sum_{N\in\mathcal D_1}N^{-\delta}\|P_Nu\|_{X^{s,b,\pm}}\lesssim_\delta \|u\|_{\mathfrak X^{s,b,\pm}},\label{X}\\
\sum_{N\in\mathcal D_1}\|P_Nu\|_{X^{s,b}}\lesssim_{\delta} \|u\|_{\mathfrak X^{s+\delta,b}}.\label{Y}
\end{align}
For $N\in\mathcal D_1$, we have
\begin{align}
\sum_{\mathcal D_1\ni N_1\sim N}\|P_{N_1}u\|_{X^{s,b,\pm}}\lesssim \|u\|_{\mathfrak X^{s,b,\pm}}\label{XX}
\end{align}
with an implicit constant which does not depend on $N$.\\[10pt]
For $k\ge1$, we have
\begin{align}
\sum_{\mathcal D_1\ni N\le k}\|P_Nu\|_{X^{s,b,\pm}}\lesssim_k\|u\|_{\mathfrak X^{s,b,\pm}}\label{XXX}.
\end{align}
\end{lemmak}
\begin{proof}
Estimate \eqref{X} is a direct consequence of the convergence of the geometric series $\sum_{N\in\mathcal D_1}N^{-\delta}$. Estimate \eqref{Y} follows from \eqref{X}.\\[10pt]
Let $C>1$ the implicit constant corresponding to $\sim$. There are $2\lfloor \log_2C\rfloor+1$ dyadic $N_1$ with $N_1\sim N$ which implies \eqref{XX}.\\[10pt]
Finally, for $k\ge1$, there are $\lfloor\log_2k\rfloor$ dyadic $N$ such that $1<N\le k$ and we obtain \eqref{XXX}.
\end{proof}
\begin{bem*}
One can show the same estimates for $Y^{s,b}$- and $\mathcal Y^{s,b}$-norms. But because of \eqref{XY}, we need these statements only for $X^{s,b}$ and $\mathfrak X^{s,b}$.
\end{bem*}
\begin{satzl}[Trilinear estimate for \boldmath{$\mathfrak X^{\frac12,-\frac12}$}\unboldmath{}]\label{L6}
Let $T\in(0,1]$, $u_j\in\mathcal S(\mathbb X\times\mathbb R)$ with $\mathrm{supp}\,u_j\subseteq\mathbb X\times(-T,T)$, $j\in\{1,2,3\}$. There exists an $\varepsilon>0$ such that
\begin{align}\label{TX}
\|\mathcal T(u_1,u_2,u_3)\|_{\mathfrak X^{\frac12,-\frac12}}\lesssim T^\varepsilon\|u_1\|_{\mathfrak X^{\frac12,\frac12}}\|u_2\|_{\mathfrak X^{\frac12,\frac12}}\|u_3\|_{\mathfrak X^{\frac12,\frac12,-}}.
\end{align}
\end{satzl}
\begin{proof}
By applying the triangle inequality, we may assume $P_{N_j}u_j\ge0$ and $P_{N_j}u_j=\chi_T(t)P_{N_j}u_j$. Take $\varepsilon\in(0,\tfrac1{32})$. We need to show the estimates
\begin{align*}
\|P_{1}\mathcal T(u_1,u_2,u_3)\|_{X^{\frac12,-\frac12}}\lesssim&\, T^\varepsilon\|u_1\|_{\mathfrak X^{\frac12,\frac12}}\|u_2\|_{\mathfrak X^{\frac12,\frac12}}\|u_3\|_{\mathfrak X^{\frac12,\frac12,-}},\\
\sup_{N>1}\|P_N\mathcal T(u_1,u_2,u_3)\|_{X^{\frac12,-\frac12}}\lesssim&\, T^\varepsilon\|u_1\|_{\mathfrak X^{\frac12,\frac12}}\|u_2\|_{\mathfrak X^{\frac12,\frac12}}\|u_3\|_{\mathfrak X^{\frac12,\frac12,-}}.
\end{align*}
Let
\begin{align*}
f_{N_j,u_j}(\xi,\tau):=&\, \langle\tau+\xi^2\rangle^{1/2}\langle\xi\rangle^{1/2}\widehat{P_{N_j}u_j}(\xi,\tau),\ j\in\{1,2\},\\
f_{N_3,u_3}(\xi,\tau):=&\, \langle\tau-\xi^2\rangle^{1/2}\langle\xi\rangle^{1/2}\widehat{P_{N_3}u_3}(\xi,\tau).
\end{align*}
We define
\begin{align*}
\Omega_\xi:=\begin{cases}\Omega'_\xi:=\mathbb R^3_\xi\ &\,\text{if }\mathbb X=\mathbb R,\\
\Omega'_\xi\cup\Omega''_{\xi}\ &\,\text{if }\mathbb X=\mathbb T,
\end{cases}
\end{align*}
where
\begin{align*}
\Omega'_\xi:=&\, \{\vec\xi\in\mathbb Z^3_\xi:\ \xi_1,\xi_2\ne \xi\},\ 
\Omega''_\xi:= \{\vec\xi\in\mathbb Z^3_\xi:\ \xi_1=\xi_2=\xi,\,\xi_3=-\xi\}
\end{align*}
for $\mathbb X=\mathbb T$ with integration with respect to the counting measure. Then
\begin{align*}
\|P_{1}\mathcal T(u_1,&\,u_2,u_3)\|_{X^{\frac12,-\frac12}}\\
=&\, \Big\|\chi_{\le1}(\xi)\sum_{N_1,N_2,N_3\in\mathcal D_1}\int_{\mathbb R^3_\tau}\int_{\Omega_\xi}M(\xi,\tau,\vec\xi,\vec\tau)\prod_{j=1}^3f_{N_j,u_j}(\xi_j,\tau_j)\,\mathrm d\vec\xi\,\mathrm d\vec\tau\Big\|_{L^2_{\xi,\tau}}
\end{align*}
and the same for $P_N\mathcal T(u_1,u_2,u_3)$ with dyadic $N>1$ by replacing $\chi_{\le1}$ with $\chi_N$.\\[10pt]
On $\Omega''_\xi$, we will only get a positive term if $N_1,N_2,N_3\lesssim 1$ and $N_1,N_2,N_3\sim N$ respectively. For $N>1$, we use estimates \eqref{tm}, \eqref{M0}-\eqref{N} and \eqref{XX} to conclude
\begin{align*}
\Big\|\chi_N(\xi)\,\sum_{N_1,N_2,N_3\in\mathcal D_1}&\int_{\mathbb R^3_\tau}\int_{\Omega''_\xi}M(\xi,\tau,\vec\xi,\vec\tau)\prod_{j=1}^3f_{N_j,u_j}(\xi_j,\tau_j)\,\mathrm d\vec\xi\,\mathrm d\vec\tau\Big\|_{L^2_{\xi,\tau}}\nonumber\\
\lesssim&\, T^\varepsilon \sum_{N_1,N_2,N_3\sim N}\|P_{N_1}u_1\|_{X^{\frac12,\frac12}}\|P_{N_2}u_2\|_{X^{\frac12,\frac12}}\|P_{N_3}u_3\|_{X^{\frac12,\frac12,-}}\nonumber\\
\lesssim&\, T^\varepsilon \|u_1\|_{\mathfrak X^{\frac12,\frac12}}\|u_2\|_{\mathfrak X^{\frac12,\frac12}}\|u_3\|_{\mathfrak X^{\frac12,\frac12,-}}
\end{align*}
with a constant which does not depend on $N$. For the small frequencies, we obtain from estimates \eqref{tm}, \eqref{M0}-\eqref{N} and \eqref{XXX} that
\begin{align*}
\Big\|\chi_{\le1}(\xi)\,\sum_{N_1,N_2,N_3\in\mathcal D_1}&\int_{\mathbb R^3_\tau}\int_{\Omega''_\xi}M(\xi,\tau,\vec\xi,\vec\tau)\prod_{j=1}^3f_{N_j,u_j}(\xi_j,\tau_j)\,\mathrm d\vec\xi\,\mathrm d\vec\tau\Big\|_{L^2_{\xi,\tau}}\nonumber\\
\lesssim&\, T^\varepsilon\sum_{N_1,N_2,N_3\lesssim1}\|P_{N_1}u_1\|_{X^{\frac12,\frac12}}\|P_{N_2}u_2\|_{X^{\frac12,\frac12}}\|P_{N_3}u_3\|_{\frac12,\frac12,-}\nonumber\\
\lesssim&\, T^\varepsilon\|u_1\|_{\mathfrak X^{\frac12,\frac12}}\|u_2\|_{\mathfrak X^{\frac12,\frac12}}\|u_3\|_{\mathfrak X^{\frac12,\frac12,-}}.
\end{align*}
We still need to consider the set $\Omega'_\xi$. The conclusion follows from
\begin{align}
\Big\|\chi_{\le1}(\xi)\sum_{N_1,N_2,N_3}\int_{\mathbb R^3_\tau}&\int_{\Omega'_\xi}M(\xi,\tau,\vec\xi,\vec\tau)\prod_{j=1}^3f_{N_j,u_j}(\xi_j,\tau_j)\,\mathrm d\vec\xi\,\mathrm d\vec\tau\Big\|_{L^2_{\xi,\tau}}\nonumber\\
\lesssim&\, T^\varepsilon\|u_1\|_{\mathfrak X^{\frac12,\frac12}}\|u_2\|_{\mathfrak X^{\frac12,\frac12}}\|u_3\|_{\mathfrak X^{\frac12,\frac12,-}},\label{AAA}\\
\sup_{N>1}\Big\|\chi_N(\xi)\sum_{N_1,N_2,N_3}\int_{\mathbb R^3_\tau}&\int_{\Omega'_\xi}M(\xi,\tau,\vec\xi,\vec\tau)\prod_{j=1}^3f_{N_j,u_j}(\xi_j,\tau_j)\,\mathrm d\vec\xi\,\mathrm d\vec\tau\Big\|_{L^2_{\xi,\tau}}\nonumber\\
\lesssim&\, T^\varepsilon\|u_1\|_{\mathfrak X^{\frac12,\frac12}}\|u_2\|_{\mathfrak X^{\frac12,\frac12}}\|u_3\|_{\mathfrak X^{\frac12,\frac12,-}}\label{BBB}.\end{align}
Estimate \eqref{AAA} can be shown by similar arguments as \eqref{BBB} just by replacing $\chi_N$ with $\chi_{\le1}$ and "$\sim N$" with "$\lesssim 1$". Hence we will only prove \eqref{BBB}.\\[10pt]
Let $N>1$ be dyadic. By symmetry, we may assume that $N_1\le N_2$. We distinguish between the cases $N_3\gg N_2$, $N_3\sim N_2$ and $N_3\ll N_2$ (taking an implicit constant greater than $8$). We write
\begin{align*}
\Big\|\sum_{N_1,N_2,N_3}\chi_N(\xi)\int_{\mathbb R^3_\tau}\int_{\Omega'_\xi}M(\xi,\tau,\vec\xi,\vec\tau)\prod_{j=1}^3f_{N_j,u_j}(\xi_j,\tau_j)\,\mathrm d\vec\xi\,\mathrm d\vec\tau\Big\|_{L^2_{\xi,\tau}}\lesssim I+II+III
\end{align*}
taking the sum over $\|\chi_N(\xi)\int_{\mathbb R^3_\tau}\int_{\Omega'_\xi}M(\xi,\tau,\vec\xi,\vec\tau)\prod_{j=1}^3f_{N_j,u_j}(\xi_j,\tau_j)\,\mathrm d\vec\xi\,\mathrm d\vec\tau\|_{L^2_{\xi,\tau}}$ with the restrictions
\begin{align*}
I:\ N_1\le N_2\ll N_3,\ II:\ N_1\le N_2\sim N_3,\ 
III:\ N_1\le N_2,\,N_2\gg N_3.
\end{align*}
\emph{Case I: $N_1\le N_2\ll N_3$. }In order to obtain a positive contribution, we may assume that $N_3\sim N$. Then resonance relation \eqref{Resona} provides
\begin{align*}
\max\{|\tau+\xi^2|,&\,|\tau_1+\xi_1^2|,|\tau_2+\xi_2^2|,|\tau_3-\xi_3^2|\}\gtrsim NN_3.
\end{align*}
Since $N,N_3\ge1$, we have $NN_3\sim \langle\xi\rangle\langle\xi_3\rangle$ which implies
\begin{align*}
|M(\xi,\tau,\vec\xi,\vec\tau)|\le&\, \frac{\langle\xi\rangle^{1/2}\langle\xi_3\rangle^{1/2}}{\langle\tau+\xi^2\rangle^{1/2}\langle\tau_1+\xi_1^2\rangle^{1/2}\langle\tau_2+\xi_2^2\rangle^{1/2}\langle\tau_3-\xi_3^2\rangle^{1/2}\langle\xi_1\rangle^{1/2}\langle\xi_2\rangle^{1/2}}\\
\lesssim&\, \sum_{j=0}^3M_j(\xi,\tau,\vec\xi,\vec\tau).
\end{align*}
By \eqref{M0}-\eqref{M3}, we obtain
\begin{align*}
I\lesssim&\, T^\varepsilon\sum_{N_1,N_2}\sum_{N_3\sim N}\|P_{N_1}u_1\|_{X^{\frac38,\frac12}}\|P_{N_2}u_2\|_{X^{\frac38,\frac12}}\|P_{N_3}u_3\|_{X^{\frac12,\frac12,-}}\\
\lesssim&\, T^\varepsilon\|u_1\|_{\mathfrak X^{\frac12,\frac12}}\|u_2\|_{\mathfrak X^{\frac12,\frac12}}\|u_3\|_{\mathfrak X^{\frac12,\frac12,-}}.
\end{align*}
\emph{Case II: $N_1\le N_2\sim N_3$. }We distinguish between 
\begin{align*}
IIa:\ &\,N_1\sim N_2,\ IIb:\ N_1\ll N_2.
\end{align*}
\emph{Case IIa: $N_1\sim N_2\sim N_3$. }For a positive contribution, we need $N_1,N_2,N_3\sim N$. Therefore, we get the desired estimate directly by \eqref{M0}-\eqref{N}.\\[10pt]
\emph{Case IIb: $N_1\ll N_2\sim N_3$. }We consider the cases 
\begin{align*}
IIb_1:\ N\lesssim N_1,\ IIb_2:\ N\gg N_1.
\end{align*}
\emph{Case $IIb_1$: $N\lesssim N_1\ll N_2\sim N_3$. }Here
\begin{align}
|M(\xi,\tau,\vec\xi,\vec\tau)|
&\,\lesssim \frac{1}{\langle\tau+\xi^2\rangle^{1/2}\langle\tau_1+\xi_1^2\rangle^{1/2}\langle\tau_2+\xi_2^2\rangle^{1/2}\langle\tau_3-\xi_3^2\rangle^{1/2}}.\label{IIb}
\end{align}
We subdivide $\Omega'_\xi= \Omega'_{\xi,+}\cup\Omega'_{\xi,-}$, where $\Omega'_{\xi,+}:=\{\vec\xi\in\Omega'_\xi:\ |\xi-\xi_1|\ge N_3^{-1/2}\}$ and $\Omega'_{\xi,-}:=\Omega'_\xi\smallsetminus\Omega'_{\xi,+}$. Then we can split $IIb_1$ into $IIb_{1,+}$ and $IIb_{1,-}$ with
\begin{align*}
IIb_{1,\pm}:= \underset{N\lesssim N_1\ll N_2\sim N_3}{\sum_{N_1,N_2,N_3}}\Big\|\chi_N(\xi)\int_{\mathbb R^3_\tau}\int_{\Omega'_{\xi,\pm}}M(\xi,\tau,\vec\xi,\vec\tau)\prod_{j=1}^3f_{N_j,u_j}(\xi_j,\tau_j)\,\mathrm d\vec\xi\,\mathrm d\vec\tau\Big\|_{L^2_{\xi,\tau}}.
\end{align*}
\emph{Case $IIb_{1,+}$: $N\lesssim N_1\ll N_2\sim N_3$ and $|\xi-\xi_1|\ge N_3^{-1/2}$. }Due to resonance relation \eqref{Resona}, we may assume that
\begin{align*}
\max\{|\tau+\xi^2|,|\tau_1+\xi_1^2|,|\tau_2+\xi_2^2|,|\tau_3-\xi_3^2|\}\gtrsim N_3^{1/2}\gtrsim N_1^{1/4}N_3^{1/4}.
\end{align*}
By \eqref{IIb}, we obtain
\begin{align}
|M(\xi,\tau,\vec\xi,\vec\tau)|\lesssim N_1^{-1/64}N_3^{-1/64}M_4(\xi,\tau,\vec\xi,\vec\tau)\label{IIb1}
\end{align}
and the conclusion follows from \eqref{N}.\\[10pt]
\emph{Case $IIb_{1,-}$: $N\lesssim N_1\ll N_2\sim N_3$ and $|\xi-\xi_1|< N_3^{-1/2}$.} In the periodic setting, we have $\xi,\xi_1\in\mathbb Z$ and $N_3^{-1/2}\le1$ implies $\xi=\xi_1$. Hence $\vec\xi\notin\Omega_\xi'$ and $\Omega'_{\xi,-}=\emptyset$.\\[10pt] 
For the non-periodic setting, we introduce 
\begin{align*}
\rho_{N_3,\xi}:=&\,\mathbbm1_{\left(\xi-N_3^{-1/2},\,\xi+N_3^{-1/2}\right)},\\
I_k:=&\,\big(kN_3^{-1/2},\,(k+1)N_3^{-1/2}\big).
\end{align*} 
Estimate \eqref{IIb} implies
\begin{align}
\Big\|\chi_N(&\xi)\int_{\mathbb R^3_\tau}\,\int_{\Omega'_{\xi,-}}M(\xi,\tau,\vec\xi,\vec\tau)\prod_{j=1}^3f_{N_j,u_j}(\xi_j,\tau_j)\,\mathrm d\vec\xi\,\mathrm\,d\vec\tau\Big\|_{L^2_{\xi,\tau}}\nonumber\\
\lesssim&\, \Big\|\langle\tau+\xi^2\rangle^{-1/2}\int_{\mathbb R^6_{\xi,\tau}}\rho_{N_3,\xi}(\xi_1)\prod_{j=1}^3\langle\xi_j\rangle^{1/2}\widehat{P_{N_j}u_j}(\xi_j,\tau_j)\,\mathrm d(\vec\xi,\vec\tau)\Big\|_{L^2_{\xi,\tau}}\nonumber\\
\le&\,\sum_{l\in\mathbb Z}\sum_{k\in\mathbb Z}\Big\|\langle\tau+\xi^2\rangle^{-1/2}\int_{\mathbb R^6_{\xi,\tau}}\mathbbm1_{I_k}(\xi_2)\mathbbm1_{I_l}(\xi_3)\rho_{N_3,\xi}(\xi_1)\nonumber\\
&\ \ \ \ \ \ \ \ \ \ \ \ \ \ \ \ \ \ \ \ \ \ \ \ \ \ \ \ \ \ \ \ \ \ \ \ \ \ \cdot\prod_{j=1}^3\langle\xi_j\rangle^{1/2}\widehat{P_{N_j}u_j}(\xi_j,\tau_j)\,\mathrm d(\vec\xi,\vec\tau)\Big\|_{L^2_{\xi,\tau}}.\label{osc}
\end{align}
For a positive contribution, we need $\xi_2\in I_k$ and $|\xi_2+\xi_3|=|\xi-\xi_1|\le N_3^{-1/2}$. Hence
\begin{align*}
\xi_3\in\big[-\xi_2-N_3^{-1/2},-\xi_2+N_3^{-1/2}\big]\subseteq\big[-(k+2)N_3^{-1/2},-(k-1)N_3^{-1/2}\big]\end{align*} and dividing the integration region into $\langle\tau_2+\xi_2^2\rangle\le\langle\tau_3-\xi_3^2\rangle$, $\langle\tau_2+\xi_2^2\rangle>\langle\tau_3-\xi_3^2\rangle$ provides
\begin{align*}
\eqref{osc}\lesssim&\,\sum_{l=0}^2\sum_{k\in\mathbb Z}\Big\|\langle\tau+\xi^2\rangle^{-1/2}\int_{\mathbb R^6_{\xi,\tau}}\mathbbm1_{I_k}(\xi_2)\mathbbm1_{I_{-(k+l)}}(\xi_3)\\
&\,\ \ \ \ \ \ \ \ \ \ \ \ \ \cdot\prod_{j=1}^3\langle\xi_j\rangle^{1/2}\widehat{P_{N_j}u_j}(\xi_j,\tau_j)\cdot\langle\tau_2+\xi_2^2\rangle^{-1/16}\langle\tau_3-\xi_3^2\rangle^{1/16} \,\mathrm d(\vec\xi,\vec\tau)\Big\|_{L^2_{\xi,\tau}}\\
&\,\ +\sum_{l=0}^2\sum_{k\in\mathbb Z}\Big\|\langle\tau+\xi^2\rangle^{-1/2}\int_{\mathbb R^6_{\xi,\tau}}\mathbbm1_{I_k}(\xi_2)\mathbbm1_{I_{-(k+l)}}(\xi_3)\\
&\,\ \ \ \ \ \ \ \ \ \ \ \ \ 
\cdot\prod_{j=1}^3\langle\xi_j\rangle^{1/2}\widehat{P_{N_j}u_j}(\xi_j,\tau_j)\cdot\langle\tau_2+\xi_2^2\rangle^{1/16}\langle\tau_3-\xi_3^2\rangle^{-1/16}\,\mathrm d(\vec\xi,\vec\tau)\Big\|_{L^2_{\xi,\tau}}\\
=:&\, A(N_1,N_2,N_3)+B(N_1,N_2,N_3).
\end{align*}
An application of estimate \eqref{L43} and the inequalities of Hölder and Cauchy-Schwarz leads to
\begin{align}
A(N_1,N_2,N_3)\lesssim&\,
\Big\|J_x^{1/2}P_{N_1}u_1\Big\|_{L^{4}_{t,x}} \sum_{l=0}^2\bigg(\sum_{k\in\mathbb Z}\Big\|P_{I_k}\big(J_x^{1/2}\Gamma^{-1/16}P_{N_2}u_2\big)\Big\|_{L^{4}_{t,x}}^2\bigg)^{1/2}
\nonumber\\
&\,\ \ \ \ \ \ \ \ \ \ \ \ \ \ \ \ \ \ \ \cdot\bigg(\sum_{k\in\mathbb Z}\big\|P_{I_{-(k+l)}}\big(J_x^{1/2}\Gamma_-^{1/16} P_{N_3}u_3\big)\big\|_{L^{4}_{t,x}}^2\bigg)^{1/2}.\label{a}
\end{align}
Bernstein's inequality, $|I_k|=N_3^{-1/2}$, $\mathrm{supp}\, u\subseteq\mathbb R\times[-T,T]$ and Strichartz estimate \eqref{LpLq} provide
\begin{align}
\Big\|P_{I_k}\big(J_x^{1/2}\Gamma^{-1/16}P_{N_2}u_2\big)\Big\|_{L^{4}_{t,x}}\lesssim&\,  N_3^{-\frac12\left(\frac12-\frac14\right)}\Big\|P_{I_k}\big(J_x^{1/2}P_{N_2}u_2\big)\Big\|_{L^4_tL^2_x}\nonumber\\
\lesssim&\, T^{1/4} N_3^{-\frac{1}{8}}\Big\|P_{I_k}\big(J_x^{1/2}\Gamma^{-1/16}P_{N_2}u_2\big)\Big\|_{L^{\infty}_tL^2_x}\nonumber\\
\lesssim&\, T^{1/4} N_3^{-\frac{1}{8}}\Big\|P_{I_k}\big(J_x^{1/2}\Gamma^{-1/16}P_{N_2}u_2\big)\Big\|_{X^{0,\frac{9}{16}}}.\label{b}
\end{align}
Applying estimate \eqref{b} to the second factor of \eqref{a} and \eqref{L41} to the other factors yields
\begin{align*}
A(N_1,N_2,N_3)\lesssim&\, T^{1/4}N_3^{-1/8}\|P_{N_1}u_1\|_{X^{\frac12,\frac{1}{2}}} \sum_{l=0}^2\Big(\sum_{k\in\mathbb Z}\|P_{I_k}P_{N_2}u_2\|_{X^{\frac12,\frac12}}^2\Big)^{1/2}\\
&\,\ \ \ \ \ \ \ \ \ \ \ \ \ \ \ \ \ \ \ \ \ \ \ \ \ \ \ \ \ \ \ \ \ \ \ \ \ \ \ \cdot\Big(\sum_{k\in\mathbb Z}\|P_{I_{-(k-l)}} P_{N_3}u_3\|_{X^{\frac12,\frac12,-}}^2\Big)^{1/2}\\
\lesssim&\,T^{1/4}N_1^{-1/16}N_3^{-{1}/{16}}\|P_{N_1}u_1\|_{X^{\frac12,\frac12}}\|P_{N_2}u_2\|_{X^{\frac12,\frac12}}\|P_{N_3}u_3\|_{X^{\frac12,\frac12,-}},
\end{align*}
where we used almost orthogonality and $N_1\lesssim N_3$.\\[10pt]
For $B(N_1,N_2,N_3)$, we obtain the same upper bound by switching between $u_2$ and $u_3$ and applying Bernstein's inequality on $\|P_{I_{-(k+l)}}P_{N_3}u_3\|_{L^4_{t,x}}$.\\[10pt]
Overall, we have shown that
\begin{align*}
II{b_{1,-}}\lesssim&\, \sum_{N_1,N_3}\sum_{N_2\sim N_3}\big(A(N_1,N_2,N_3)+B(N_1,N_2,N_3)\big)\\
\lesssim&\, T^{1/4}\|u_1\|_{\mathfrak X^{\frac12,\frac12}}\|u_2\|_{\mathfrak X^{\frac12,\frac12}}\|u_3\|_{\mathfrak X^{\frac12,\frac12,-}}.
\end{align*}
\emph{Case $IIb_2$: $N_1\ll N$ and $N_1\ll N_2\sim N_3$. }We have $|\xi_1+\xi_3|\sim N_3$, $|\xi_2+\xi_3|=|\xi-\xi_1|\sim N$ and by resonance relation \eqref{Resona}, we can suppose that
\begin{align*}
\max\{|\tau+\xi^2|,|\tau_1+\xi_1^2|,|\tau_2+\xi_2^2|,|\tau_3-\xi_3^2|\}
\gtrsim N_3N.
\end{align*}
Since $N_3\sim N_2$, the desired estimate can be concluded by the same arguments as in case $I$.\\[10pt]
\emph{Case III: $N_1\le N_2$ and $N_2\gg N_3$. }We further distinguish between 
\begin{align*}
IIIa_1:&\ N_1\gg N_3,\ N_1\ll N_2,\\IIIa_2:&\ N_1\gg N_3,\ N_1\sim N_2,\\
IIIb:&\ N_1\sim N_3,\ 
IIIc:\ N_1\ll N_3. 
\end{align*}
\emph{Case $IIIa_1$: $N_2\gg N_1\gg N_3$. }We will only obtain positive terms, if $N_2\sim N$ and, by resonance relation \eqref{Resona}, we have $\max\{|\tau+\xi^2|,|\tau_1+\xi_1^2|,|\tau_2+\xi_2^2|,|\tau_3-\xi_3^2|\}\gtrsim N_1N_2$. This means
\begin{align}
|M(\xi,\tau,\vec\xi,\vec\tau)|&\,\lesssim \frac{1}{\langle\tau+\xi^2\rangle^{1/2}\langle\tau_1+\xi_1^2\rangle^{1/2}\langle\tau_2+\xi_2^2\rangle^{1/2}\langle\tau_3-\xi_3^2\rangle^{1/2}}\nonumber\\
&\,\lesssim N_1^{-1/16}N_2^{-1/16}M_4(\xi,\tau,\vec\xi,\vec\tau)\label{IIIa11}.
\end{align}
Since $N_2^{-1/16}\lesssim N_3^{-1/16}$, we obtain the conclusion with arguments similar to case $II{b_{1,+}}$.\\[10pt]
\emph{Case $IIIa_2$: $N_1\sim N_2\gg N_3$. }For a positive contribution, we need $N_2\gtrsim N$ and 
\begin{align*}
\max\{|\tau+\xi^2|,|\tau_1+\xi_1^2|,|\tau_2+\xi_2^2|,|\tau_3-\xi_3^2|\}
\gtrsim N_1N_2\sim N_2^2\gtrsim NN_2
\end{align*}
by resonance relation \eqref{Resona}. Hence $|M(\xi,\tau,\vec\xi,\vec\tau)|\,\lesssim  \sum_{j=1}^3M_j(\xi,\tau,\vec\xi,\vec\tau)$. Since $N_2^{-1/8}\lesssim N_2^{-1/16}N_3^{-1/16}$, we get the desired estimate by similar arguments as in $I$.\\[10pt]
\emph{Case IIIb: $N_2\gg N_1\sim N_3$. }We can assume that $N_2\sim N$ and
\begin{align*}
\max\{|\tau+\xi^2|,|\tau_1+\xi_1^2|,|\tau_2+\xi_2^2|,|\tau_3-\xi_3^2|\}\gtrsim N_2|\xi-\xi_2|\ge N_2^{1/2},
\end{align*}
if $|\xi-\xi_2|\ge N_2^{-1/2}$ -- otherwise we use similar arguments to case $IIb_{1,-}$. We obtain
\begin{align*}
|M(\xi,\tau,\vec\xi,\vec\tau)|\lesssim&\, \frac{1}{\langle\tau+\xi^2\rangle^{1/2}\langle\tau_1+\xi_1^2\rangle^{1/2}\langle\tau_2+\xi_2^2\rangle^{1/2}\langle\tau_3-\xi_3^2\rangle^{1/2}}\nonumber\allowdisplaybreaks[1]\\
\lesssim&\, N_2^{-1/32}M_4(\xi,\tau,\vec\xi,\vec\tau)
\end{align*}
which implies the conclusion by arguing as in case $IIb_{1,+}$.\\[10pt]
\emph{Case IIIc: $N_2\gg N_3\gg N_1$. }Here $N_2\sim N$ and 
\begin{align*}
\max\{|\tau+\xi^2|,|\tau_1+\xi_1^2|,|\tau_2+\xi_2^2|,|\tau_3-\xi_3^2|\}\gtrsim N_2N_3.
\end{align*}
Hence $|M(\xi,\tau,\vec\xi,\vec\tau)|\lesssim \sum_{j=1}^3M_j(\xi,\tau,\vec\xi,\vec\tau)$ and we obtain the desired estimate as argumented in case $I$.
\end{proof}
\begin{satzl}[Trilinear estimate for \boldmath{$\mathcal Y^{\frac12,-1}$}\unboldmath{}]\label{Tri}
Let $T\in(0,1]$, $u_j\in\mathcal S(\mathbb X\times\mathbb R)$ and $\mathrm{supp}\,u_j\subseteq\mathbb X\times[-T,T],\ j\in\{1,2,3\}$. There is an $\varepsilon>0$ such that
\begin{align}
\|\mathcal T(u_1,u_2,u_3)\|_{\mathcal Y^{\frac12,-1}}\lesssim T^\varepsilon\|u_1\|_{\mathfrak X^{\frac12,\frac12}}\|u_2\|_{\mathfrak X^{\frac12,\frac12}}\|u_3\|_{\mathfrak X^{\frac12,\frac12,-}}.\label{TY}
\end{align}
\end{satzl}
\begin{proof}By an application of the triangle inequality, we may assume that $P_{N_j}u_j\ge0$ and $P_{N_j}u_j=\chi_T(t)P_{N_j}u_j$. As in the main part of the previous proof, we will omit the case of small frequencies and focus on $|\xi|$ of order $N>1$. We have
\begin{align}
\|P_{N}\mathcal T&\,(u_1,u_2,u_3)\|_{Y^{\frac12,-1}}\nonumber\\
\le&\, \sum_{N_1,N_2,N_3}\Big\|\chi_N(\xi)\int_{\mathbb R^3_\tau}\int_{\mathbb Y^3_\xi}\tilde M(\xi,\tau,\vec\xi,\vec\tau)\prod_{j=1}^3f_{N_j,u_j}(\xi_j,\tau_j)\,\mathrm d\vec\xi\,\mathrm d\vec\tau\Big\|_{L^2_\xi L^1_\tau}.\label{L2L1}
\end{align}
We subdivide $\tilde M$ into $\tilde M_0$, $\tilde M_1$, $\tilde M_2$, $\tilde M_3$, $\tilde M_4$ and adapt the ideas of \cite[Thm. 4.2]{Herr} to our setting: For $\gamma>0$, $\xi\in\mathbb Y$ and $w\colon\mathbb Y\times\mathbb R\to\mathbb C$ such that $w(\xi,\,\cdot\,)\in L^2(\mathbb R)$, Hölder's inequality and the transformation theorem lead to
\begin{align}
\big\|\langle\tau\pm\xi^2\rangle^{-\frac12-\gamma}w(\xi,\tau)\big\|_{L^1_\tau}\lesssim_{\gamma}\big\|w(\xi,\tau)\big\|_{L^2_{\tau}}\label{CSU}.
\end{align}
By definition of $\tilde M_0$ and Young's convolution inequality, we get
\begin{align}
\bigg\|&\chi_N(\xi)\int_{\mathbb R^3_\tau}\int_{\mathbb Y^3_\xi}\tilde M_0(\xi,\tau,\vec\xi,\vec\tau)\prod_{j=1}^3f_{N_j,u_j}(\xi_j,\tau_j)\,\mathrm d\vec\xi\,\mathrm d\vec\tau\bigg\|_{L^2_\xi L^1_\tau}\nonumber\\
\lesssim&\, \bigg\|\int_{\mathbb Y^3_\xi}\langle\xi\rangle^{-\frac12+\delta}\prod_{j=1}^2\Big\|\frac{f_{N_j,u_j}(\xi_j,\tau_j)}{\langle\xi_j\rangle^{\frac12}\langle\tau_j+\xi_j^2\rangle^{\frac\delta2}}\Big\|_{L^2_{\tau_j}}\cdot\Big\|\frac{f_{N_3,u_3}(\xi_3,\tau_3)}{\langle\xi_3\rangle^{\frac12-3\delta}\langle\tau_3-\xi_3^2\rangle^{\frac\delta2}}\Big\|_{L^2_{\tau_3}}\mathrm d\vec\xi\bigg\|_{L^2_\xi}.\label{os}
\end{align}
Let 
\begin{align*}
g_{N_j,u_j}(\xi,\tau):=&\, \langle\tau+\xi^2\rangle^{-\delta/2}f_{N_j,u_j}(\xi,\tau),\ j\in\{1,2\},\\
g_{N_3,u_3}(\xi,\tau):=&\, \langle\tau-\xi^2\rangle^{-\delta/2}f_{N_3,u_3}(\xi,\tau).
\end{align*} 
In the sequel, we choose $\delta=\tfrac{1}{24}$. Then, by Young's and Hölder's inequality,
\begin{align}
\eqref{os}\lesssim&\,\bigg\|\int_{\mathbb Y^3_\xi}\langle\xi\rangle^{-3/8}\prod_{j=1}^2\langle\xi_j\rangle^{-1/2}\cdot\langle\xi_3\rangle^{-3/8}\prod_{j=1}^3\Big\|g_{N_j,u_j}(\xi_j,\tau_j)\Big\|_{L^2_{\tau_j}}\mathrm d\vec\xi\bigg\|_{L^2_\xi}\nonumber\\
\lesssim&\, N_1^{-3/16} N_2^{-3/16} N_3^{-1/16}\prod_{j=1}^3
\Big\|g_{N_j,u_j}(\xi_j,\tau_j)\Big\|_{L^2_{\xi_j,\tau_j}}\nonumber\\
=&\, N_1^{-3/16} N_2^{-3/16} N_3^{-1/16}\prod_{j=1}^2\Big\|P_{N_j}u_j\Big\|_{X^{\frac12,\frac{23}{48}}}\cdot\Big\|P_{N_3}u_3\Big\|_{X^{\frac12,\frac{23}{48},-}}.\nonumber
\end{align}
Hence
\begin{align}
\Big\|\chi_N(\xi)\int_{\mathbb R^3_\tau}\int_{\mathbb Y^3_\xi}\tilde M_0(\xi,\tau,\vec\xi,\vec\tau)&\,\prod_{j=1}^3f_{N_j,u_j}(\xi_j,\tau_j)\,\mathrm d\vec\xi\,\mathrm d\vec\tau\Big\|_{L^2_{\xi}L^1_\tau}\nonumber\\
\lesssim&\, T^\varepsilon \|P_{N_1}u_1\|_{X^{\frac{5}{16},\frac{1}{2}}}\|P_{N_2}u_2\|_{X^{\frac{5}{16},\frac{1}{2}}}\|u_3\|_{X^{\frac{7}{16},\frac{1}{2},-}}.\label{tM0}
\end{align}
By definition of $\tilde M_1$ and estimate \eqref{CSU}, we have
\begin{align*}
\bigg\|\chi_N(\xi)\int_{\mathbb R^3_\tau}&\,\int_{\mathbb Y^3_\xi}\tilde M_1(\xi,\tau,\vec\xi,\vec\tau)\prod_{j=1}^3f_{N_j,u_j}(\xi_j,\tau_j)\,\mathrm d\vec\xi\,\mathrm d\vec\tau\bigg\|_{L^2_\xi L^1_\tau}\nonumber\\
\lesssim&\, \bigg\|\chi_N(\xi)\int_{\mathbb R^3_\tau}\int_{\mathbb Y^3_\xi}\langle\tau+\xi^2\rangle^{1/8} M_1(\xi,\tau,\vec\xi,\vec\tau)\prod_{j=1}^3f_{N_j,u_j}(\xi_j,\tau_j)\,\mathrm d\vec\xi\,\mathrm d\vec\tau\bigg\|_{L^2_{\xi,\tau}}\\
\lesssim&\,\Big\|\Gamma^{1/2}P_{N_1}u_1\cdot{P_{N_2}u_{2}}\cdot J_{x}^{1/2}P_{N_3}u_3\Big\|_{X^{0,-\frac38}}.
\end{align*}
Applying \eqref{m11} leads to
\begin{align}
\Big\|\chi_N(\xi)\int_{\mathbb R^3_\tau}\int_{\mathbb Y^3_\xi}\tilde M_1&\,(\xi,\tau,\vec\xi,\vec\tau)\prod_{j=1}^3f_{N_j,u_j}(\xi_j,\tau_j)\,\mathrm d\vec\xi\,\mathrm d\vec\tau\Big\|_{L^2_{\xi}L^1_\tau}\nonumber\\
\lesssim&\, T^\varepsilon\|P_{N_1}u_1\|_{X^{\frac38,\frac12}}\|P_{N_2}u_2\|_{X^{\frac38,\frac12}}\|P_{N_3}u_3\|_{X^{\frac12,\frac12,-}}.\label{tM1}
\end{align}
By changing the first two factors, one can show analogously
\begin{align}
\Big\|\chi_N(\xi)\int_{\mathbb R^3_\tau}\int_{\mathbb Y^3_\xi}\tilde M_2&(\xi,\tau,\vec\xi,\vec\tau)\prod_{j=1}^3f_{N_j,u_j}(\xi_j,\tau_j)\,\mathrm d\vec\xi\,\mathrm d\vec\tau\Big\|_{L^2_{\xi}L^1_\tau}\nonumber\\
\lesssim&\, T^\varepsilon\|P_{N_1}u_1\|_{X^{\frac38,\frac12}}\|P_{N_2}u_2\|_{X^{\frac38,\frac12}}\|P_{N_3}u_3\|_{X^{\frac12,\frac12,-}}.\label{tM2}
\end{align}
For $\tilde M_3$, we conclude by \eqref{CSU} that
\begin{align*}
\bigg\|\chi_N(\xi)\int_{\mathbb R^3_\tau}&\,\int_{\mathbb Y^3_\xi}\tilde M_3(\xi,\tau,\vec\xi,\vec\tau)\prod_{j=1}^3f_{N_j,u_j}(\xi_j,\tau_j)\,\mathrm d\vec\xi\,\mathrm d\vec\tau\bigg\|_{L^2_\xi L^1_\tau}\nonumber\\
\lesssim&\, \bigg\|\int_{\mathbb R^3_\tau}\int_{\mathbb Y^3_\xi}\langle\tau+\xi^2\rangle^{-7/16}\langle\tau_3-\xi_3^2\rangle^{1/2}\langle\xi_3\rangle^{1/2} \prod_{j=1}^3\widehat{P_{N_j}u_j}(\xi_j,\tau_j)\,\mathrm d\vec\xi\,\mathrm d\vec\tau\bigg\|_{L^2_{\xi,\tau}}\\
\lesssim&\, \big\|{P_{N_1}u_1}\cdot{P_{N_2}u_2}\cdot J_x^{1/2}\Gamma_-^{1/2}P_{N_3}u_3\big\|_{X^{0,-\frac7{16}}}
\end{align*}
and \eqref{m33} implies 
\begin{align}
\Big\|\chi_N(\xi)\int_{\mathbb R^3_\tau}\int_{\mathbb Y^3_\xi}\tilde M_3&\,(\xi,\tau,\vec\xi,\vec\tau)\prod_{j=1}^3f_{N_j,u_j}(\xi_j,\tau_j)\,\mathrm d\vec\xi\,\mathrm d\vec\tau\Big\|_{L^2_{\xi}L^1_\tau}\nonumber\\
\lesssim&\, T^\varepsilon \|P_{N_1}u_1\|_{X^{\frac38,\frac12}}\|P_{N_2}u_2\|_{X^{\frac38,\frac12}}\|P_{N_3}u_3\|_{X^{\frac12,\frac12,-}}
.\label{tM3}
\end{align}
By definition of $\tilde M_4$ and estimate \eqref{CSU}, we have
\begin{align*}
\bigg\|\chi_N(\xi)\int_{\mathbb R^3_\tau}&\,\int_{\mathbb Y^3_\xi}\tilde M_4(\xi,\tau,\vec\xi,\vec\tau)\prod_{j=1}^3f_{N_j,u_j}(\xi_j,\tau_j)\,\mathrm d\vec\xi\,\mathrm d\vec\tau\bigg\|_{L^2_\xi L^1_\tau}\\
\lesssim&\, \bigg\|\int_{\mathbb R^3_\tau}\int_{\mathbb Y^3_\xi}\langle\tau+\xi^2\rangle^{-1/2} M_4(\xi,\tau,\vec\xi,\vec\tau)\prod_{j=1}^3f_{N_j,u_j}(\xi_j,\tau_j)\,\mathrm d\vec\xi\,\mathrm d\vec\tau\bigg\|_{L^2_\xi L^1_\tau}\\
\lesssim&\, \big\|J_x^{1/2}\Gamma^{1/16}{P_{N_1}u_1}\cdot J_x^{1/2}\Gamma^{1/16}{P_{N_2}u_2}\cdot J_x^{1/2}\Gamma^{1/16}_-{P_{N_3}u_3}\big\|_{X^{0,-\frac{7}{16}}}.
\end{align*}
From \eqref{m44}, we conclude
\begin{align}
\Big\|\chi_N(\xi)\int_{\mathbb R^3_\tau}\int_{\mathbb Y^3_\xi}\tilde M_4&\,(\xi,\tau,\vec\xi,\vec\tau)\prod_{j=1}^3f_{N_j,u_j}(\xi_j,\tau_j)\,\mathrm d\vec\xi\,\mathrm d\vec\tau\Big\|_{L^2_{\xi}L^1_\tau}\nonumber\\
\lesssim&\, T^\varepsilon \|P_{N_1}u_1\|_{X^{\frac12,\frac12}}\|P_{N_2}u_2\|_{X^{\frac12,\frac12}}\|P_{N_3}u_3\Big\|_{X^{\frac12,\frac12,-}}.\label{tN}
\end{align}
Now, we show estimate \eqref{L2L1}. As in the proof of theorem \ref{L6}, it suffices to consider $\Omega'_\xi$ instead of $\mathbb Y^3_\xi$. We denote
\begin{align*}
\sum_{N_1,N_2,N_3}\Big\|&\,\chi_N(\xi)\int_{\mathbb R^3_\tau}\int_{\Omega'_\xi}\tilde M(\xi,\tau,\vec\xi,\vec\tau)\prod_{j=1}^3f_{N_j,u_j}(\xi_j,\tau_j)\,\mathrm d\vec\xi\,\mathrm d\vec\tau\Big\|_{L^2_\xi L^1_\tau}\\
\lesssim&\, I+IIa+II{b_{1,+}}+II{b_{1,-}}+IIb_2+IIIa_1+IIIa_2+IIIb+IIIc,
\end{align*}
where $I,...,IIIc$ refer to the same cases as in the proof of theorem \ref{L6}.\\[10pt]
\emph{Case I: $N_1\le N_2\ll N_3$. }We have $N\sim N_3$ and 
\begin{align*}
\max\{|\tau+\xi^2|,|\tau_1+\xi_1^2|,|\tau_2+\xi_2^2|,|\tau_3-\xi_3^2|\}\sim NN_3.
\end{align*}
Since $\tilde M(\xi,\tau,\vec\xi,\vec\tau)=\langle\tau+\xi^2\rangle^{-1/2}M(\xi,\tau,\vec\xi,\vec\tau)$, it follows from the same arguments as in the proof of theorem \ref{L6} that
\begin{align}
|\tilde M(\xi,\tau,\vec\xi,\vec\tau)|
\lesssim&\, \langle\tau+\xi^2\rangle^{-1/2}\sum_{j=0}^3M_j(\xi,\tau,\vec\xi,\vec\tau)\nonumber\\
=&\, \langle\tau+\xi^2\rangle^{-1/2}M_0(\xi,\tau,\vec\xi,\vec\tau)+\sum_{j=1}^3\tilde M_j(\xi,\tau,\vec\xi,\vec\tau)\nonumber\\
\lesssim&\, \sum_{j=0}^3\tilde M_j(\xi,\tau,\vec\xi,\vec\tau),\label{tmop}
\end{align}
where we used resonance relation \eqref{Resona} in the last step as follows: For $(\vec\xi,\vec\tau)\in A_0(\xi,\tau)$, we have 
\begin{align*}
\langle\tau+\xi^2\rangle\gtrsim NN_3\sim\langle\xi\rangle\langle\xi_3\rangle.
\end{align*}
Hence, for any $\delta\in(0,\tfrac16)$, it holds
\begin{align}
\langle\tau+\xi^2\rangle^{1/2}\gtrsim&\, 
\langle\xi\rangle^{\frac12-3\delta}\langle\xi_3\rangle^{\frac12-3\delta}\langle\tau_1+\xi_1^2\rangle^{\delta}\langle\tau_2+\xi_2^2\rangle^{\delta}\langle\tau_3-\xi_3^2\rangle^{\delta}.\label{Party}
\end{align}
Finally, by \eqref{tM0}-\eqref{tM3}, we obtain
\begin{align*}
I\lesssim&\, T^\varepsilon\sum_{N_1,N_2}\sum_{N_3\sim N}\|P_{N_1}u_1\|_{X^{\frac38,\frac12}}\|P_{N_2}u_2\|_{X^{\frac38,\frac12}}\|P_{N_3}u_3\|_{X^{\frac12,\frac12,-}}\\
\lesssim&\, T^\varepsilon\|u_1\|_{\mathfrak X^{\frac12,\frac12}}\|u_2\|_{\mathfrak X^{\frac12,\frac12}}\|u_3\|_{\mathfrak X^{\frac12,\frac12,-}}.
\end{align*}
%
\emph{Case IIa: $N_1\sim N_2\sim N_3$. }Since $N_1,N_2,N_3\sim N$, this follows from \eqref{tM0}-\eqref{tN}.\\[10pt]
\emph{Case $IIb_{1,\pm}$: $N\sim N_1\ll N_2\sim N_3$. }Since $|\tilde M(\xi,\tau,\vec\xi,\vec\tau)|= \langle\tau+\xi^2\rangle^{-1/2}|M(\xi,\tau,\vec\xi,\vec\tau)|$, estimate \eqref{IIb} implies
\begin{align*}
|\tilde M(\xi,\tau,\vec\xi,\vec\tau)|
\lesssim&\,\frac{1}{\langle\tau+\xi^2\rangle\langle\tau_1+\xi_1^2\rangle^{1/2}\langle\tau_2+\xi_2^2\rangle^{1/2}\langle\tau_3-\xi_3^2\rangle^{1/2}}.
\end{align*}
First, consider $\vec\xi\in \Omega'_{\xi,+}$ which means $|\xi-\xi_1|\ge N_3^{-1/2}$. As in \eqref{IIb1}, we have
\begin{align}
|\tilde M(\xi,\tau,\vec\xi,\vec\tau)|\lesssim N_1^{-1/64}N_3^{-1/64}\langle\tau+\xi^2\rangle^{-1/2}M_4(\xi,\tau,\vec\xi,\vec\tau)\label{61}.
\end{align}
An application of \eqref{CSU} and Strichartz estimates \eqref{L43},\,\eqref{L41} together with Hölder's inequality provides 
\begin{align*}
\Big\|\chi_N(\xi)\int_{\mathbb R^3_\tau}&\,\int_{\Omega'_{\xi,+}}\langle\tau+\xi^2\rangle^{-1/2}M_4(\xi,\tau,\vec\xi,\vec\tau)\prod_{j=1}^3f_{N_j,u_j}(\xi_j,\tau_j)\,\mathrm d\vec\xi\,\mathrm d\vec\tau\Big\|_{L^2_\xi L^1_\tau}\\
\lesssim&\, \big\|J_x^{1/2}\Gamma^{1/16}P_{N_1}u_1\cdot J_x^{1/2}\Gamma^{1/16}P_{N_2}u_2\cdot J_x^{1/2}\Gamma^{1/16}_-P_{N_3}u_3\big\|_{X^{0,-\frac{13}{32}}}\\
\lesssim&\, \big\|J_x^{1/2}\Gamma^{1/16}P_{N_1}u_1\big\|_{L^4_{t,x}}\big\|J_x^{1/2}\Gamma^{1/16}P_{N_2}u_2\big\|_{L^4_{t,x}}\big\|J_x^{1/2}\Gamma^{1/16}_-P_{N_3}u_3\big\|_{L^4_{t,x}}\\
\lesssim&\, 
T^\varepsilon\|P_{N_1}u_1\|_{X^{\frac12,\frac12}}\|P_{N_2}u_2\|_{X^{\frac12,\frac12}}\|P_{N_3}u_3\|_{X^{\frac12,\frac12,-}}
\end{align*}
%
and we obtain the desired estimate by \eqref{61}, \eqref{X} and \eqref{XX}.\\[10pt]
Now, let $\vec\xi\in\Omega'_{\xi,-}$ which means $|\xi-\xi_1|<N_3^{1/2}$. Estimate \eqref{CSU} implies
\begin{align*}
\Big\|\chi_N(\xi)&\int_{\mathbb R^3_\tau}\int_{\Omega'_{\xi,-}}\tilde M(\xi,\tau,\vec\xi,\vec\tau)\prod_{j=1}^3f_{N_j,u_j}(\xi_j,\tau_j)\,\mathrm d\vec\xi\,\mathrm d\vec\tau\Big\|_{L^2_\xi L^1_\tau}\\
\lesssim&\, \Big\|\chi_N(\xi)\int_{\mathbb R^3_\tau}\int_{\Omega'_{\xi,-}}\langle\tau+\xi^2\rangle^{-1/2}M(\xi,\tau,\vec\xi,\vec\tau)\prod_{j=1}^3f_{N_j,u_j}(\xi_j,\tau_j)\,\mathrm d\vec\xi\,\mathrm d\vec\tau\Big\|_{L^2_\xi L^1_\tau}\\
\lesssim&\, \Big\|\chi_N(\xi)\int_{\mathbb R^3_\tau}\int_{\Omega'_{\xi,-}}\langle\tau+\xi^2\rangle^{1/16}M(\xi,\tau,\vec\xi,\vec\tau)\prod_{j=1}^3f_{N_j,u_j}(\xi_j,\tau_j)\,\mathrm d\vec\xi\,\mathrm d\vec\tau\Big\|_{L^2_{\xi,\tau}}.
\end{align*}
Finally, we can argue as done in the proof of theorem \ref{L6}: Note that estimate \eqref{L43} used in \eqref{a} is also valid for $b=-\tfrac7{16}$ instead of $b=-\tfrac{1}{2}$. Hence, we can ignore the extra term $\langle\tau+\xi^2\rangle^{1/16}$, here.\\[10pt]
\emph{Case $IIb_2$: $N_1\ll N$ and $N_1\ll N_2\sim N_3$. }We may assume that
\begin{align*}\max\{|\tau+\xi^2|,|\tau_1+\xi_1^2|,|\tau_2+\xi_2^2|,|\tau_3-\xi_3^2|\}\gtrsim NN_3\end{align*} 
which provides the conclusion by arguing as in case $I$.\\[10pt] 
\emph{Case $IIIa_1$: $N_2\gg N_1\gg N_3$. }We have $N_2\sim N$ and
\begin{align*}
\max\{|\tau+\xi^2|,|\tau_1+\xi_1^2|,|\tau_2+\xi_2^2|,|\tau_3-\xi_3^2|\}\gtrsim N_2N_1. 
\end{align*}
This means
\begin{align*}
|\tilde M(\xi,\tau,\vec\xi,\vec\tau)|=&\, \langle\tau+\xi^2\rangle^{-1/2}|M(\xi,\tau,\vec\xi,\vec\tau)|\\
\lesssim&\, \langle\tau+\xi^2\rangle^{-1/2}N_1^{-1/16}N_2^{-1/16}M_4(\xi,\tau,\vec\xi,\vec\tau)
\end{align*}
by \eqref{IIIa11} and we get the desired estimate using $N_2^{-1/16}\le N_3^{-1/16}$ as in case $IIb_{1,+}$.\\[10pt]
\emph{Case $IIIa_2$: $N_1\sim N_2\gg N_3$. }We need $N_2\gtrsim N$ and
\begin{align*}
\max\{|\tau+\xi^2|,|\tau_1+\xi_1^2|,|\tau_2+\xi_2^2|,|\tau_3-\xi_3^2|\}\gtrsim N_2^2.
\end{align*}
Hence $|\tilde M(\xi,\tau,\vec\xi,\vec\tau)|\lesssim \sum_{j=0}^3\tilde M_j(\xi,\tau,\vec\xi,\vec\tau)$ by \eqref{tmop}. Since $N_3^{-1/8}\lesssim N_2^{-1/16}N_3^{-1/16}$, the conclusion follows from \eqref{tM0}-\eqref{tM3}.\\[10pt]
\emph{Case $IIIb$: $N_2\gg N_1\sim N_3$. }Here, $N_2\sim N$ and 
\begin{align*}
\max\{|\tau+\xi^2|,|\tau_1+\xi_1^2|,|\tau_2+\xi_2^2|,|\tau_3-\xi_3^2|\}\gtrsim N_2^{1/2},
\end{align*}
if $|\xi-\xi_1|\ge N_2^{-1/2}$ -- otherwise we use the ideas of $IIb_{1,-}$. Hence
\begin{align*}
|\tilde M(\xi,\tau,\vec\xi,\vec\tau)|\lesssim N_2^{-1/32}\langle\tau+\xi^2\rangle^{-1/2}M_4(\xi,\tau,\vec\xi,\vec\tau)
\end{align*}
which implies the desired estimate as in case $IIb_{1,+}$.\\[10pt]
\emph{Case $IIIc$: $N_2\gg N_3\gg N_1$. }We may assume $N_2\sim N$ and 
\begin{align*}
\max\{|\tau+\xi^2|,|\tau_1+\xi_1^2|,|\tau_2+\xi_2^2|,|\tau_3-\xi_3^2|\}\gtrsim N_2N_3.
\end{align*}
Therefore
\begin{align*}
|\tilde M(\xi,\tau,\vec\xi,\vec\tau)|
\lesssim \langle\tau+\xi^2\rangle^{-1/2}M_0(\xi,\tau,\vec\xi,\vec\tau)+\sum_{j=1}^3M_j(\xi,\tau,\vec\xi,\vec\tau)
\lesssim \sum_{j=0}^3\tilde M_j(\xi,\tau,\vec\xi,\vec\tau),
\end{align*}
where the last step can be seen by using $\langle\tau+\xi^2\rangle\gtrsim N_2N_3\sim NN_3$ for $(\vec\xi,\vec\tau)\in A_0(\xi,\tau)$ and a calculation as in \eqref{Party}. We obtain the conclusion by arguments similar to case $I$ using $N_2^{-1/8}\lesssim N_2^{-1/16}N_3^{-1/16}$.
\end{proof}
\begin{korl}[Trilinear estimate for \boldmath{$s\ge\tfrac12$}\unboldmath{}]\label{kor}
Let $s\ge\frac12$, $\delta>0$, $T\in(0,1]$, $u_j\in\mathcal S(\mathbb X\times\mathbb R)$ such that $\mathrm{supp}\,u_j\subseteq\mathbb X\times[-T,T]$, $j\in\{1,2,3\}$. Then, for some $\varepsilon>0$,
\begin{align*}
\|\mathcal T(u_1,u_2,\overline{u_3})\|_{\mathfrak X^{s,-\frac12}\cap\mathcal Y^{s,-1}}\lesssim&\, T^\varepsilon\sum_{k=1}^3\|u_k\|_{\mathfrak X^{s,\frac12}}\prod_{\underset{j\ne k}{j=1}}^3\|u_j\|_{\mathfrak X^{\frac12,\frac12}}.
\end{align*}
\end{korl}
\begin{proof}
As argumented before, we focus on frequencies $|\xi|$ of order $N>1$. Since $\langle\xi\rangle^{s-\frac12}
\lesssim \sum_{k=1}^3\langle\xi_k\rangle^{s-\frac12}$, we have
\begin{align*}
\|P_N\mathcal T(u_1,u_2,\overline{u_3})\|_{X^{s,-\frac12}}
\lesssim \Big\|P_N\mathcal T\big(&J_x^{s-\frac12}u_1,u_2,\overline{u_3}\big)\Big\|_{X^{\frac12,-\frac12}}\\
&\,+ \left\|P_N\mathcal T\big(u_1,J_x^{s-\frac12}u_2,\overline{u_3}\big)\right\|_{X^{\frac12,-\frac12}}\\
&\,+ \left\|P_N{\mathcal T\big(u_1,u_2,J_x^{s-\frac12}\partial_x\overline{u_3}\big)}\right\|_{X^{\frac12,-\frac12}}.
\end{align*}
Estimate $\eqref{TX}$ and $\|\overline u\|_{X^{s,b,-}}=\|u\|_{X^{s,b}}$ imply
\begin{align*}
\|\mathcal T(u_1,u_2,\overline{u_3})\|_{\mathfrak X^{s,-\frac12}}\lesssim T^\varepsilon \sum_{k=1}^3\|u_k\|_{\mathfrak X^{s,\frac12}}\prod_{\underset{j\ne k}{j=1}}^3\|u_j\|_{\mathfrak X^{\frac12,\frac12}}.
\end{align*}
Replacing \eqref{TX} by \eqref{TY} provides the same upper bound for $\|\mathcal T(u_1,u_2,\overline{u_3})\|_{\mathcal Y^{s,-1}}$.
\end{proof}
\section{Multilinear Estimate}
In this section, we consider the polynomial terms $\mathcal Q(v)$ and $|v|^{2k}v$, $k\in\mathbb N_0$. The absence of derivatives in these terms leads to a less technical proof.
\begin{satzk}
Let $s\ge\frac12$, $\delta>0$, $k\in\mathbb N_0$, $T\in(0,1]$, $u_j\in\mathcal S(\mathbb X\times\mathbb R)$ satisfying $\mathrm{supp}\,u_j\subseteq\mathbb X\times[-T,T]$, $j\in\mathbb N_{\le k+1}$. There is an $\varepsilon>0$ such that
\begin{align}
\Big\|\prod_{j=1}^{k+1}u_j\Big\|_{\mathfrak X^{s,-\frac38-\delta}\cap\mathcal Y^{s,-1}}\lesssim&\, T^{\varepsilon}\sum_{l=1}^{k+1}\|u_l\|_{\mathfrak X^{s,\frac12,\pm}}\prod_{\underset{j\ne l}{j=1}}^{k+1}\|u_j\|_{\mathfrak X^{\frac12,\frac12,\pm}}\label{multi}
\end{align}
and in particular
\begin{align}
\big\|\mathcal Q(u_1,\overline {u_2},u_3,\overline{u_4},u_5)\big\|_{\mathfrak X^{s,-\frac38-\delta}\cap\mathcal Y^{s,-1}}\lesssim T^{\varepsilon}\sum_{l=1}^5\|u_l\|_{\mathfrak X^{s,\frac12}}\prod_{\underset{j\ne l}{j=1}}^5\|u_j\|_{\mathfrak X^{\frac12,\frac12}}\label{27}.
\end{align}
\end{satzk}
\begin{proof}
By triangle inequality, we may assume $P_{N_j}u_j\ge0$ and $P_{N_j}u_j=\chi_T(t)P_{N_j}u_j$. Hence, estimate \eqref{27} is a direct consequence of \eqref{multi}. According to \eqref{XY}, we have $X^{s,-\frac38-\delta}\hookrightarrow Y^{s,-1}$ for $\delta\in(0,\tfrac18)$. Therefore, it suffices to handle the $\mathfrak X^{s,-\frac38-\delta}$-norm. \\[10pt]
For $k=0$, we can conclude estimate \eqref{multi} from $X^{s,\frac12}\hookrightarrow X^{s,-\frac38-\delta}$ for $\delta>0$ and $X^{s,-\frac38-\delta}\hookrightarrow Y^{s,-1}$ for $\delta\in(0,\tfrac18)$, compare \eqref{XY}.\\[10pt]
Now, let $k\ge1$. As before, we focus on $N>1$. Applying $\langle\xi\rangle^s\lesssim \sum_{l=1}^{k+1}\langle\xi_l\rangle^s$ leads to 
\begin{align*}
\Big\|P_N\prod_{j=1}^{k+1}u_j\Big\|_{X^{s,-\frac38-\delta}}\lesssim\sum_{l=1}^{k+1}\bigg\|P_N\Big(J_x^{s}u_l\prod_{\underset{j\ne l}{j=1}}^{k+1}u_j\Big)\bigg\|_{X^{0,-\frac38-\delta}}.
\end{align*}
Let $\varepsilon\in(0,\tfrac12)$, $l\in\mathbb N_{\le k+1}$ and $B_l:=\{(N_1,...,N_{k+1})\in\mathcal D_1^{k+1}:\ N_j\ll N_l\ \forall\,j\ne l\}$ (with an implicit constant greater than $4k$). We decompose
\begin{align*}
\bigg\|P_N\bigg(&J_x^{s}u_l\prod_{\underset{j\ne l}{j=1}}^{k+1}u_j\bigg)\bigg\|_{X^{0,-\frac38-\delta}}\\
&\le\, \sum_{(N_1,...,N_{k+1})\in B_l}\bigg\|\chi_N(\xi)\langle\tau+\xi^2\rangle^{-\frac38-\delta}\bigg(P_{N_k}J_x^{s}u_l\prod_{\underset{j\ne l}{j=1}}^{k+1}P_{N_j}u_j\bigg)\widehat{\textcolor[rgb]{1,1,1}{u}}(\xi,\tau)\bigg\|_{L^2_{\xi,\tau}}\\
&\,\ \ \ +\sum_{(N_1,...,N_{k+1})\in B_l^c}\bigg\|\chi_N(\xi)\langle\tau+\xi^2\rangle^{-\frac38-\delta}\bigg(P_{N_k}J_x^{s}u_l\prod_{\underset{j\ne l}{j=1}}^{k+1}P_{N_j}u_j\bigg)\widehat{\textcolor[rgb]{1,1,1}{u}}(\xi,\tau)\bigg\|_{L^2_{\xi,\tau}}\\
&=: I+II.
\end{align*}
\emph{Case I: } We need $N_l\sim N$ for a positive contribution. From Strichartz estimate \eqref{L43}, Hölder's inequality and estimate \eqref{11}, we conclude
\begin{align*}
I\lesssim&\, \sum_{N_l\sim N}{\sum_{N_j,\,j\ne l}}\bigg\|P_{N_l}J_x^{s}u_l\prod_{\underset{j\ne l}{j=1}}^{k+1}P_{N_j}u_j\bigg\|_{L^{4/3}_{t,x}}\\
\lesssim&\,\sum_{N_l\sim N}\big\|P_{N_l}u_l\big\|_{X^{s,0,\pm}}\underset{j\ne l}{\prod_{j=1}^{k+1}}\sum_{N_j}\big\|P_{N_j}u_j\big\|_{X^{\frac12-\frac{1}{4k},\frac12,\pm}}
\lesssim T^{\varepsilon}\|u_l\|_{\mathfrak X^{s,\frac12,\pm}}\underset{j\ne l}{\prod_{j=1}^{k+1}}\|u_j\|_{\mathfrak X^{\frac12,\frac12,\pm}}.
\end{align*}
\emph{Case II: }For $(N_1,..,N_{k+1})\in B_l^c$, there is a $j_l^\ast\in\mathbb N_{\le{k+1}}\smallsetminus\{l\}$ such that $N_l\lesssim N_{j_l^\ast}$. This means
\begin{align*}
II\lesssim&\, \sum_{N_j,j\ne l}\sum_{N_l\lesssim N_{j_l^\ast}}\Big\|P_{N_l}J_x^su_l\underset{j\ne l}{\prod_{j=1}^{k+1}}P_{N_j}u_j\Big\|_{L^{4/3}_{t,x}}\\
\lesssim&\, \sum_{N_j,j\ne l}\sum_{N_l\lesssim N_{j_l^\ast}}\big\|P_{N_{j^\ast_l}}u_{j^\ast_l}\big\|_{X^{\frac12-\frac{1}{8k},\frac12,\pm}}\big\|P_{N_l}u_l\big\|_{X^{s-\frac{1}{8k},0,\pm}}\underset{j\ne l,j_l^\ast}{\prod_{j=1}^{k+1}}\|u_j\|_{\mathfrak X^{\frac12,\frac12,\pm}}\\
\lesssim&\, T^{\varepsilon}\|u_l\|_{\mathfrak X^{s,\frac12,\pm}}\underset{j\ne l}{\prod_{j=1}^{k+1}}\|u_j\|_{\mathfrak X^{\frac12,\frac12,\pm}}
\end{align*}
using the dual Strichartz estimate \eqref{L43} in the first step, Hölder's inequality and estimate \eqref{11} in the second step.
\end{proof}
\section{Local well-posedness}
By the standard contraction mapping principle, we obtain the following local well-posedness result for the gauge equivalent problem:
\begin{satzk}\label{WT}
Let $s\ge\tfrac12$, $k\in\mathbb N_0$, $r>0$ and $B_r:=\{v_0\in B^s_{2,\infty}(\mathbb X):\ \|v_0\|_{B^{s}_{2,\infty}(\mathbb X)}<r\}$. For any $v_0\in B_r$, there is a $T=T(r)>0$ such that the Cauchy problem \eqref{TP2} has a unique solution $v\in\mathcal Z^s_T$. The flow map
\begin{align*}
F\colon B_r\to \mathcal C([-T,T],B^s_{2,\infty}(\mathbb X)),\ v_0\mapsto v
\end{align*}
is Lipschitz continuous.
\end{satzk}
We can conclude local well-posedness for equation \eqref{G2} (i.e. prove theorem \ref{W}) by the same strategy as in Herr \cite{Herr}: The Gauge transformation is a locally bilipschitz homeomorphism on $\mathcal C([-T,T],B^s_{2,\infty})$ which can be shown by an application of Sobolev's multiplication theorem for Besov spaces:
\begin{align}
\|f_1f_2\|_{B^s_{2,\infty}}\lesssim \|f_1\|_{B^{s_1}_{2,\infty}}\|f_2\|_{B^{s_2}_{2,\infty}}\label{smult}
\end{align}
for $f_1\in B^{s_1}_{2,\infty}$, $f_2\in B^{s_2}_{2,\infty}$, $s\ge0$, $s_1,s_2\ge s$, $s_1+s_2-s>\tfrac12$. A proof for $H^s$ instead of $B^s_{2,\infty}$ can be found for example in \cite[Corollary 1.1.12]{HerrDiss}. With trivial modifications, one can show \eqref{smult} in a similar way by localizing frequencies on $N_1,N_2$ and considering the cases $N_1\sim N_2$, $N_1\ll N_2$. We obtain the statement of theorem \ref{W} by establishing $M_{s,T}:=G^{-1}(\mathcal Z^s_T)$ in the non-periodic setting and $M_{s,T}:=\mathcal G^{-1}(\mathcal Z^s_T)$ in the periodic setting.
\subsection*{Acknowledgement} This is essentially the author's master thesis finished in September 2016. The author would like to thank his supervisor Prof. Sebastian Herr for many helpful thoughts and giving a lot of valuable suggestions.
\bibliographystyle{plain}
\bibliography{Literaturliste2}
\end{document}